\documentclass[10.9pt,a4paper, reqno]{amsart}
\usepackage{a4wide}

\usepackage[utf8]{inputenc}
\usepackage[english]{babel}

\usepackage{amsfonts,amssymb,amsmath}
\usepackage[T1]{fontenc}
\usepackage{lmodern}
\usepackage[hidelinks]{hyperref}

\usepackage{tikz}
\usepackage[all]{xy}
\usepackage{enumitem}
\usepackage{color}
\makeatletter
\newcommand{\mylabel}[2]{#2\def\@currentlabel{#2}\label{#1}}
\usetikzlibrary{calc, matrix, arrows, cd}

\newcommand{\bsm}{\left(\begin{smallmatrix}}
\newcommand{\esm}{\end{smallmatrix}\right)}

\newtheorem{theorem}{Theorem}[section]

\newtheorem{lemma}[theorem]{Lemma}
\newtheorem{proposition}[theorem]{Proposition}

\theoremstyle{definition}

\newtheorem{remark}[theorem]{Remark}
\newtheorem*{remark*}{Remark}

\newtheorem{notation}[theorem]{Notation}
\newtheorem{convention}[theorem]{Convention}
\newtheorem{construction}[theorem]{Construction}
\newtheorem{claim}{Claim}
\newtheorem*{claim*}{Claim}

\newcommand{\ol}{\overline}

\newcommand{\Max}{\operatorname{Max}}

\newcommand{\La}{\Lambda}
\newcommand{\GL}{\operatorname{GL}}
\newcommand{\SL}{\operatorname{SL}}
\newcommand{\coker}{\operatorname{coker}}
\newcommand{\im}{\operatorname{im}}

\newcommand{\Z}{\mathbb{Z}}
\newcommand{\Q}{\mathbb{Q}}

\newcommand{\F}{\mathbb{F}}

\newcommand{\s}{\mathbb{S}}

\title{Unknotting orientable surfaces}
\author{Anthony Conway}
\address{The University of Texas at Austin, Austin TX 78712}
\email{anthony.conway@austin.utexas.edu}

\begin{document}
\maketitle

\begin{abstract}
It is shown that every locally flatly embedded genus~$g \in \{1,2\}$ surface
in the~$4$-sphere with knot group~$\Z$ is unknotted.
The same proof establishes that any two genus~$g \in \{1,2\}$ surfaces in~$D^4$ with knot group $\Z$ and common boundary an Alexander polynomial one knot are isotopic rel. boundary.
In previous work, the author and Powell reduced such unknotting problems to a question concerning the cancellation of $(-t)$-quadratic forms over $\Z[t^{\pm 1}]$,  which was solved in genus $g \geq 3$ using work of Bass.
In genus $g=2$, we observe that the same proof goes through using further work of Bass.
In genus $g=1,$ the result instead follows from a statement in commutative algebra which was proved with the assistance of AI.
 Combined with earlier work of Freedman on locally flat spheres with knot group $\Z$,  this shows that a locally flatly embedded orientable surface in~$S^4$ is unknotted if and only if its knot group is $\Z$.
\end{abstract}

\section{Introduction}
\label{sec:statement}

A question in the study of knotted surfaces,  sometimes referred to as the unknotting conjecture, asks whether an embedded surface in~$S^4$ with cyclic knot group is necessarily unknotted; see e.g.~\cite[Question 4.29(a)]{K3}.
Here, the \emph{knot group} of an embedded surface~$F \subset S^4$ refers to the group~$\pi_1(S^4 \setminus F)$.
This question is of interest both in the smooth and topological categories.
In the smooth category,  the answer is negative for nonorientable surfaces~\cite{FinashinKreckViro} and remains unknown for orientable surfaces.
In the topological category, the conjecture posits that an orientable (resp. nonorientable) surface~$F \subset S^4$ with knot group~$\Z$ (resp.  $\Z_2$) is necessarily unknotted.
For nonorientable surfaces, this has been disproved~\cite{ConwayGalvin}, whereas the orientable case remains open.

From now on,  we work in the topological category,  with locally flat (closed and connected) surfaces and topological isotopies,  and review what is known about the unknotting conjecture in more detail.
Nonorientable genus~$h$ surfaces in~$S^4$ with knot group~$\Z_2$ are unknotted when either the Euler number~$e$ satisfies~$|e| \neq 2h$ or when~$h \leq 5$~\cite{ConwayOrsonPowell} (with the case~$h=1$ originally due to Lawson~\cite{Lawson} and the cases~$h=4,5$ requiring~\cite{Pencovitch}); on the other hand, for every~$h \geq 9$, there are nontrivial knotted surfaces in~$S^4$ with nonorientable genus~$h$ and knot group~$\Z_2$~\cite{ConwayGalvin}.
Orientable genus~$g$ surfaces in~$S^4$ with knot group~$\Z$ are unknotted when~$g=0$~\cite{Freedman,FreedmanQuinn} and when~$g\geq 3$~\cite{ConwayPowell}.
The cases~$g=1,2$ remain open; see~\cite[Theorem 1.9]{ConwayPowell} and~\cite{JuhaszPowell} for evidence that the unknotting conjecture might also hold in these cases.

\medbreak

The present article settles the unknotting conjecture in genus $g=1,2$.

\begin{theorem}
\label{thm:ZTori}
A genus $g\in \{1,2\}$ surface~$F \subset S^4$ with knot group~$\pi_1(S^4\setminus F)\cong \Z$ is unknotted.
\end{theorem}

In particular, we note that if $F \subset S^4$ is a smoothly embedded torus with four critical points, then $F$ is topologically unknotted.
This answers the topological version of Question (i) in the remark following~\cite[Problem 4.29]{K3}.
The result was previously known when the~$2$-component link~$L:=F \cap S^3$ is assumed to have identically zero Alexander polynomial~\cite[Theorem 1.4]{JuhaszPowell}.

Combining Theorem~\ref{thm:ZTori} with~\cite{Freedman,ConwayPowell} leads to the following result which resolves the orientable case of~\cite[Question 4.29(a)]{K3}.
\begin{theorem}
An orientable surface in~$S^4$ is unknotted if and only if its knot group is $\Z$.
\end{theorem}

The same proof also yields the following result concerning surfaces in the $4$-ball.
\begin{theorem}
\label{thm:D4}
Let $K \subset S^3$ be an Alexander polynomial one knot.
Any two genus~$g$ surfaces in~$D^4$ with knot group $\Z$ and boundary $K$ are isotopic rel. boundary.
\end{theorem}

When $g \neq 1,2$,  Theorem~\ref{thm:D4} follows from~\cite[Theorem 1.2]{ConwayPowell}, with the result for discs first appearing in~\cite{ConwayPowellDiscs}.
Our contribution to Theorem~\ref{thm:D4} is the case where $g \in \{1,2\}$.

\subsection{Proof outline}
\label{sub:ProofOutline}
Prior work of the author and Powell~\cite[Theorem 1.4]{ConwayPowell} shows that a genus $g$ surface $F \subset S^4$ with knot group $\Z$ is unknotted if and only if the equivariant intersection form~$\lambda_{X_F}$ of its exterior is isometric to (the hermitian form underlying)~$H^{\oplus g}$, where $H$ denotes the following matrix with values in the Laurent polynomial
 ring~$\Lambda:=\Z[t^{\pm 1}]$:
$$
H:=\begin{pmatrix}
0&t-1 \\ t^{-1}-1&0
\end{pmatrix}.
$$
When $g\geq 3$,  it is known that~$\lambda_{X_F} \cong H^{\oplus g}$~\cite[Theorem 7.4]{ConwayPowell}.
The argument hinges on the following result of Bass~\cite[Corollary IV.3.6]{Bass}; see~\cite[Proposition 7.3]{ConwayPowell} for the present formulation.
Here,  we endow the ring~$\Lambda:=\Z[t^{\pm 1}]$ with the involution~$-\colon \Lambda \to \Lambda$ induced by~$t \mapsto~t^{-1}$ and, given a central unit~$\varepsilon \in \Lambda$ with $\varepsilon \ol \varepsilon=1$,  we refer the reader to~\cite[Section 1.1]{RanickiExact} for the definition of an~$\varepsilon$-quadratic form (see also~\cite[Section 7.1]{ConwayPowell}).
We write~$H_{\varepsilon}(\Lambda)=\left(\Lambda^2,\left[\bsm 0&1 \\ 0&0 \esm \right] \right)$ for the standard hyperbolic $\varepsilon$-quadratic form.

\begin{proposition}[{Bass~\cite[Corollary IV.3.6]{Bass}}]
\label{prop:ConwayPowell73}
Let~$\varepsilon \in \Lambda$ be a central unit with~$\varepsilon\ol{\varepsilon}=1$ and let~$(V,\theta),(V',\theta')$ be~$\varepsilon$-quadratic forms over~$\Lambda$.
Assume that for some~$n \geq 0$ there is an isometry
$$ (V,\theta)\oplus H_{\varepsilon}(\Lambda)^{\oplus n} \cong (V',\theta') \oplus H_{\varepsilon}(\Lambda)^{\oplus n}.$$
If the Witt index of~$(V',\theta')$ satisfies~$\operatorname{ind}(V',\theta') \geq 3$, then there is an isometry
$$ (V,\theta) \cong (V',\theta').$$
\end{proposition}

Here, the \emph{Witt index}~$\operatorname{ind}(V,\theta)$ of an~$\varepsilon$-quadratic form is the largest integer~$k$ such that a subform of~$(V,\theta)$ is isometric to the hyperbolic form~$H_{\varepsilon}(\Lambda)^{\oplus k}=\left(\Lambda^{2k},\left[\bsm 0&1 \\ 0&0 \esm\right] ^{\oplus k}\right)$.

\begin{remark}
\label{rem:CancellationEnough}
For $\varepsilon=-t,$ if the hypothesis~$\operatorname{ind}(V',\theta') \geq 3$ could be replaced by~$\operatorname{ind}(V',\theta') \geq g$ for~$g \in \{1,2\}$, then~\cite[proof of Theorem 7.4]{ConwayPowell} would establish the unknotting conjecture in genus~$g\in \{1,2\}$.
This is how Theorem~\ref{thm:ZTori} is proved in genus $g=2$,  as we now outline.
\end{remark}

Bass permits the condition~$\operatorname{ind}(V',\theta') \geq 3$ to be replaced by~$\operatorname{ind}(V',\theta') \geq 2$ provided~$\GL_3(\Lambda)$ acts transitively on the set of unimodular vectors of~$\Lambda^3$, and~$d_\Lambda \leq 1$~\cite[IV.3.2(b) and Corollary~IV.3.6]{Bass}.
Here,  we recall that a vector~$v \in \Lambda^n$ is \emph{unimodular} if its coordinates generate~$\La$ (see e.g.~\cite[IV.3.1]{Bass}) but we refer to~\cite[IV.3.1]{Bass} for the definition of the integer~$d_\Lambda$: we will not need its definition,  only the upper bound provided in~\cite[Proposition~IV.3.9]{Bass}.
We nevertheless note that in the notation of Bass, $\Lambda$ does not refer to the ring  over which the forms are considered but instead to the \emph{form parameter} which, in the present set-up,  is the set~$\{x-\varepsilon \ol x \colon x \in \Lambda\}.$

Specialising to $\varepsilon=-t$ leads to the following cancellation result.

\begin{theorem}[{Bass~\cite[Corollary IV.3.6 and Proposition IV.3.9]{Bass}}]
\label{thm:g=2}
Let~$(V,\theta),(V',\theta')$ be~$(-t)$-quadratic forms over~$\Lambda$.
Assume that for some~$n \geq 0$ there is an isometry
$$ (V,\theta)\oplus H_{-t}(\Lambda)^{\oplus n} \cong (V',\theta') \oplus H_{-t}(\Lambda)^{\oplus n}.$$
If~$\operatorname{ind}(V',\theta') \geq 2$,  then there is an isometry
$$ (V,\theta) \cong (V',\theta').$$
\end{theorem}
\begin{proof}
We claim that the group~$\GL_3(\Lambda)$ acts transitively on the set of
unimodular vectors in~$\Lambda^3$.
Let~$v\in \Lambda^3$ be unimodular.  
By definition, there is a~$\Lambda$-linear map~$f \colon \Lambda^3\to \Lambda$ with~$f(v)=~1$.  
Therefore~$\Lambda^3=\Lambda v\oplus\ker(f)$ and so the module~$\ker(f)$ is finitely generated projective of rank~$2$.
Finitely generated projective~$\Lambda$-modules are free (see e.g.~\cite[Chapter V, Corollary~4.12]{Lam}) so in fact~$\ker(f)\cong \Lambda^2.$
 Choose a basis~$v_2,v_3$ of~$\ker(f)$. 
Since~$v,v_2,v_3$ is a basis of~$\Lambda^3$,  there is an element of~$\GL_3(\Lambda)$ that carries the first standard basis vector to~$v$.  
Hence all unimodular vectors lie in the same~$\GL_3(\Lambda)$-orbit.  
This concludes the proof of the claim.

We claim that~$d_\Lambda \leq 1$.
Using~$I(\Lambda) \subset \Lambda$ to denote the ideal generated by all elements of the form~$a-\ol a$ with~$a \in \Lambda$, Bass's~\cite[Proposition~IV.3.9]{Bass} with~$R=A=\Lambda$ and $\lambda=-t$ states
\[
 d_\Lambda \leq
 \sup\left(0,
 \dim\Max\bigl(\Lambda/(I(\Lambda)+\Lambda(1+t))\bigr)\right).
\]
Here $\Max$ denotes the set of maximal ideals, appropriately topologised.

We assert that~$I(\Lambda)=(t-t^{-1})$.
The $\supset$ inclusion clearly holds, and so we focus on the $\subset$ inclusion.
Since for any $a \in \Lambda$,  the element $a-\ol a$ has no constant factor, this further reduces to proving that~$t^n-t^{-n} \in (t-t^{-1})$, i.e. that $t-t^{-1}$ divides $t^n-t^{-n}$ for every $n \geq 2$.
This now follows by noting that~$t^n-t^{-n}
 =(t-t^{-1})(t^{n-1}+t^{n-3}+\cdots+t^{-(n-1)}),$ yielding the assertion.
Next, we note that~$I(\Lambda)+\Lambda(1+t)=\Lambda(1+t)$: we have~$ t-t^{-1}=t^{-1}(t-1)(t+1)$ so $I(\Lambda)\subseteq \Lambda(1+t)$ and the reverse inclusion clearly holds.

The previous two paragraphs imply that~$\Lambda/(I(\Lambda)+\Lambda(1+t))=\Lambda/(1+t)\cong \Z$ and so
$$d_\Lambda \leq
 \sup\left(0,
 \dim\Max(\Z) \bigr)\right)=1.$$
This concludes the proof of the second claim.

The theorem now follows from these two claims together with~\cite[Corollary IV.3.6]{Bass}.
\end{proof}

To the best of the author's knowledge,  the methods of Bass cannot be used to prove Theorem~\ref{thm:g=2} when~$\operatorname{ind}(V',\theta') = 1$; see the hypotheses in~\cite[IV.3.2]{Bass}.
When $g=1,$ Theorem~\ref{thm:ZTori} is instead a consequence of the following result which is the main technical theorem of this article,  and which was obtained with the assistance of~AI.

\begin{theorem}\label{thm:main}
If a size~$2$ nondegenerate hermitian matrix~$A$ over~$\Z[t^{\pm 1}]$ presents the same module as~$H$, then it is congruent to~$H$, i.e. $\overline{P}^TAP=H$ for some invertible matrix $P \in GL_2(\Z[t^{\pm 1}])$.
\end{theorem}

We note that Theorem~\ref{thm:main} is false without the assumption that $A$ be hermitian.
For example, the matrix~$B=\bsm 0&t-1 \\ t^{-1}-1 &t-1 \esm$ is nondegenerate,  presents the same module as~$H$,
but is not congruent to~$H$.
If it were,  say~$\overline{P}^TBP=H$ for some invertible matrix~$P$, then taking the conjugate transpose would give~$\overline{P}^T(B-\overline{B}^T)P=0$ and thus~$B=\overline{B}^T$, a contradiction.

\medbreak

We describe how Theorem~\ref{thm:ZTori} follows from Theorem~\ref{thm:main} combined with the cancellation results of Bass~\cite{Bass} (namely Theorem~\ref{thm:g=2}) and the prior work of the author and Powell~\cite{ConwayPowell}.
\begin{proof}[Proof of Theorems~\ref{thm:ZTori} and~\ref{thm:D4} assuming Theorem~\ref{thm:main}]
When $g=2$,  the result follows from Remark~\ref{rem:CancellationEnough} and Theorem~\ref{thm:g=2}, so we focus on the case where $g=1$.
Let $F \subset S^4$ be a torus with knot group~$\Z$.
The equivariant intersection form of the exterior of~$F \subset S^4$ is represented by a size $2$ nondegenerate hermitian matrix over~$\Z[t^{\pm 1}]$ that presents the same module as $H$~\cite[Lemmas~3.2,5.5 and Sections~6-7]{ConwayPowell}.
Theorem~\ref{thm:main} implies that this form is isometric to the hermitian form determined by~$H$, i.e. to the equivariant intersection form of the exterior of the unknotted torus~$U$; see e.g.~\cite[Lemma~6.1]{ConwayPowell}.
It then follows from~\cite[Theorem~1.4]{ConwayPowell} that $F$ is unknotted.
The proof of Theorem~\ref{thm:D4} is entirely analogous but,  in the final step,  apply~\cite[Theorem~1.3]{ConwayPowell} instead of~\cite[Theorem~1.4]{ConwayPowell}.
\end{proof}

\subsection{Outline of the proof of Theorem~\ref{thm:main}. }
Endow the ring~$\Lambda:=\Z[t^{\pm 1}]$ with the involution~$-\colon \Lambda \to \Lambda $ induced by~$t \mapsto t^{-1}$.
Proposition~\ref{prop:A1=0Automatic} shows that if a nondegenerate size~$2$ hermitian matrix~$A$ over~$\Lambda$ presents the same module as~$H$,  then~$A$ must be of the form
\begin{equation*}
 A = \begin{pmatrix} c(t-1)(t^{-1}-1) & a(t-1) \\ \overline{a}(t^{-1}-1) & b(t-1)(t^{-1}-1) \end{pmatrix} 
 \end{equation*}
for some~$a, b, c \in \Lambda$ with~$b = \overline{b},c = \overline{c}, a(1)=\pm 1$ and~$a\overline{a}-bc(t-1)(t^{-1}-1)=1$. 
 
The congruence problem underlying Theorem~\ref{thm:main}, namely whether it is possible to find a matrix~$P \in GL_2(\Lambda)$ such that $\overline{P}^THP=A$,  is quadratic in the columns of $P$.
The first step of the proof (carried out in Section~\ref{sec:cocycle}) consists of reformulating the equation into one that is linear in the columns of the unknown matrix.
For this,  write
$$S:=\begin{pmatrix} \ol a&b(t-1)\\c(t^{-1}-1)&a\end{pmatrix},  \quad
D:=\begin{pmatrix} t&0\\0&1\end{pmatrix}$$
 and consider the following additive involution on the set $M_2(\Lambda)$ of size $2$ matrices over $\Lambda$:
\begin{align*}
\sigma \colon M_2(\Lambda) &\to M_2(\Lambda) \\
X &\mapsto D\ol X D^{-1}.
\end{align*}
Section~\ref{sec:cocycle} shows (via a fairly straightforward matricial calculation) that without loss of generality one can assume that $a(1)=1$, and that Theorem~\ref{thm:main} follows if there is a~$Q\in\SL_2(\La)$ such that
$$\sigma(Q)=SQ.$$
A further computation shows that a (possibly singular) matrix $Q=(q_1,q_2)$ satisfies this equation if and only if its columns respectively satisfy~$D\ol q_1=tSq_1$ and~$D\ol q_2=Sq_2$.
Setting~$u:=(t-1)(t^{-1}-1)$ and $R:=\Z[u]$, it is therefore natural to consider the solution $R$-modules 
\begin{align*}
&E_1=\{v\in\La^2: D\ol v=tSv\}=\{ v \in \Lambda^2 : t^{-1}S^{-1}D\ol v=v \} \\
&E_2=\{w\in\La^2:D\ol w=Sw\}=\{w\in\La^2 : S^{-1}D\ol w=w\}.
\end{align*}
Note that the assignments~$q_1 \mapsto t^{-1}S^{-1}D\ol q_1$ and~$q_2 \mapsto S^{-1}D\ol q_2$ are additive and~$R$-linear but not~$\Lambda$-linear.
Thus, $E_1$ and $E_2$ are $R$-modules but not $\Lambda$-modules.

Showing that the~$E_i$ are nontrivial produces nontrivial (but possibly singular) matrices~$Q$ that satisfy~$\sigma(Q)=SQ$.
Section~\ref{sec:global} shows that~$E_1$ and~$E_2$ are actually~$R$-free of rank~$2$ and proves that it is then possible to construct a matrix~$Q$ that additionally satisfies~$\det(Q)=1$.
This argument, as well as the fact that~$E_i \cong R^2$ for~$i=1,2$,  is obtained by establishing the corresponding results after localising at every maximal ideal~$\mathfrak{m} \subset R$.
The key ingredient in both cases is the resolution of a localised version of the equation~$\sigma(Q)=SQ$: Section~\ref{sec:Local} shows that for every maximal ideal~$\mathfrak{m} \subset R$, there is a matrix~$Q_{\mathfrak m}\in\GL_2(\La_{\mathfrak m})$ that satisfies~$\sigma(Q_{\mathfrak m})=SQ_{\mathfrak m}.$
Setting~$f:=(1+t)(1+t^{-1})=4-u$, this is proved by considering four cases (which are shown to be exhaustive in Remark~\ref{rem:Casework}):
\begin{enumerate}
\item Case 1: $u \in \mathfrak{m}$ and $2 \notin \mathfrak{m}$.
In this case,  the matrix~$1+S$ is shown to be invertible over $\Lambda_{\mathfrak{m}}$ and a solution is given by~$Q_{\mathfrak{m}}:=(1+S)^{-1}$.
\item Case 2: $f \in \mathfrak{m}$ and $2 \notin \mathfrak{m}$.
When~$a \neq -1 \in R_{\mathfrak{m}}/\mathfrak{m}R_{\mathfrak{m}}$,  one can again take~$Q_{\mathfrak{m}}:=(1+S)^{-1}$.
When~$a = -1 \in R_{\mathfrak{m}}/\mathfrak{m}R_{\mathfrak{m}}$,  one sets~$W:=\bsm 0&-1 \\ 1&0  \esm$ and takes $Q_{\mathfrak{m}}:=W(I+\sigma(W)^{-1}SW)^{-1}$.
\item Case 3: $\mathfrak{m}=(2,u)$.
Here one shows that $E_1^{\mathfrak{m}}:=\{ v \in \Lambda^2_{\mathfrak{m}} : t^{-1}S^{-1}D\ol v=v \}$ contains a unimodular element $q_1$ that can then be completed to the required~$Q_{\mathfrak{m}}:=(q_1,q_2)$.
\item Case 4: $uf \notin \mathfrak{m}$.
Here, one shows that~$E_2^{\mathfrak{m}}:=\{w\in \La_{\mathfrak m}^2: S^{-1}D\ol w=w\}$ is free of rank two over~$R_{\mathfrak m}$, and picking a basis $m_1,m_2 \in E_2^{\mathfrak{m}}$, one takes~$Q_{\mathfrak{m}}:=((1+t^{-1})m_1\ \ m_2)$.
\end{enumerate}   
The details of these arguments can be found in Sections~\ref{sub:1}-\ref{sub:4}.

\begin{remark}
Given~$g \geq 1$, it is natural to ask whether the proof of Theorem~\ref{thm:main} adapts to show that a nondegenerate size~$2g$ hermitian matrix~$A$ that presents the same module as $H^{\oplus g}$ is necessarily congruent to~$H^{\oplus g}$.
Just as in the~$g=1$ case,  it is possible to factor~$A$ as~$A=SH^{\oplus g}$ for some matrix~$S$ and to reduce the problem to solving an equation of the form~$\sigma(Q)=SQ$ for some involution~$\sigma$.
The key point is that this involution is additive if and only if~$g=1$
 and this property of $\sigma$ is used frequently and crucially during the proof, e.g. to study the equation~$\sigma(Q)=SQ$ one column at a time, leading to linear equations.
\end{remark}

\subsection*{AI Disclosure.}

In genus $g=2$,  conversations with ChatGPT 5.6 were used to summarise and reformulate notation-heavy portions of~\cite{Bass}. 
During these conversations,  the AI pointed the author to~\cite[Proposition IV.3.9]{Bass} and suggested that, combined with  it~\cite[IV.3.2(b)]{Bass}, it could be used to improve the range in Proposition~\ref{prop:ConwayPowell73}.

In order to record how AI was used in genus~$g=1,$ we first describe how the unknotting conjecture can be reformulated algebraically.
This reformulation (and its proof) did not involve AI but provides context for the prompt that led to Theorem~\ref{thm:main}.
Endow the ring~$\Lambda:=\Z[t^{\pm 1}]$ with the involution induced by~$t \mapsto t^{-1}$, and let~$\Q(t)$ be the field of fractions of~$\Lambda$.
Given a size~$n$ hermitian nondegenerate matrix~$A$ over~$\Lambda$,  set 
$$M(A):=\Lambda^n/A\Lambda^n,$$
and consider the hermitian \emph{boundary linking form} of $A$:
$$
\ell_A \colon M(A) \times M(A) \to \Q(t)/\Lambda, \quad
([x],[y]) \mapsto \overline{x}^TA^{-1}y.
$$
The unknotting conjecture can now be reformulated as follows.

\begin{proposition}
\label{prop:Equiv}
Given an integer $g \geq 0$, the following assertions are equivalent.
\begin{enumerate}
\item Every genus $g$ surface~$F \subset S^4$ with knot group $\Z$ is unknotted.
\item  Let~$A$ be a nondegenerate size~$2g$ hermitian matrix over~$\Lambda$ such that~$A|_{t=1}$ is the zero matrix.
If~$A$ and~$H^{\oplus g}$ have isometric boundary linking forms,  $\ell_A \cong \ell_{H^{\oplus g}}$,  then~$A$ and~$H^{\oplus g}$ are congruent over~$\Lambda$.
\end{enumerate}
\end{proposition}
\begin{proof}
We start with three recollections.
First,~\cite[Theorem 1.11]{ConwayPiccirilloPowell} implies that a size~$2g$ nondegenerate hermitian matrix~$A$ over $\Lambda$ arises as the equivariant intersection form of the exterior of a~$\Z$-surface~$F \subset S^4$ of genus~$g$ if and only if~$A|_{t=1}=0$ and~$\ell_A\cong \ell_{H^{\oplus g}}$.
Secondly,~\cite[Theorem~1.4]{ConwayPowell} implies that closed~$\Z$-surfaces in simply-connected~$4$-manifolds are determined up to equivalence by the equivariant intersection form of their surface exteriors.
Thirdly, the equivariant intersection form of the genus~$g$ unknot~$F \subset S^4$ is represented by~$H^{\oplus g}$~\cite[Lemma 6.1]{ConwayPowell}.

We prove that~$(1) \Rightarrow (2)$ holds.
Apply~\cite[Theorem 1.11]{ConwayPiccirilloPowell}  to represent the matrix~$A$ as the equivariant intersection form of a~$\Z$-surface~$F \subset S^4$.
By~$(1)$,  this surface is unknotted and it therefore follows that~$A \cong H^{\oplus g}$.
To see that~$(2) \Rightarrow (1)$,  note that~$(2)$ implies that the equivariant intersection form of a~$\Z$-surface~$F \subset S^4$ necessarily agrees with that of the unknotted surface from which it follows, using~\cite[Theorem~1.4]{ConwayPowell}, that~$F$ is unknotted.
\end{proof}

The author prompted several models on the algebraic reformulation of the unknotting conjecture described in Proposition~\ref{prop:Equiv}.
Google's Gemini developed an algebraic proof (not included in this article) that, when~$g=1$,  the condition~$A|_{t=1}=0$ and the linking form condition~$\ell_A \cong \ell_H$ follow from the requirement that~$M(A) \cong M(H)$.
This led the author to prompts related to the statement of Theorem~\ref{thm:main}.
An initial proof was provided by ChatGPT 5.6.
The author then verified,  reorganised and revised that proof; the exposition in the present article is the author’s.

\medbreak
Finally, we note that at the time of writing, feeding the statement of Theorem~\ref{thm:ZTori} directly into (publicly available) AI models does not lead to a proof of the result.
The same is true for the algebraic reformulation of Proposition~\ref{prop:Equiv} as well as for attempts at generalising the cancellation result from Proposition~\ref{prop:ConwayPowell73} (without specialising to~$\varepsilon=-t$ or directly requesting a thorough search of the work of Bass~\cite{Bass}) to a statement involving $\operatorname{ind}(V',\theta') \geq g$ for $g \in \{1,2\}$.

\subsection*{A personal note}
One can debate whether this paper should exist given its heavy use of AI.
Opinions will differ but I would like to conclude this introduction by outlining the decision to post the article to the arXiv.
I have been thinking about the unknotting conjecture for over 8 years.
As described in Proposition~\ref{prop:Equiv},  prior work from~\cite{ConwayPowell,ConwayPiccirilloPowell} shows that the problem in the orientable case can be reduced to a question about $2 \times 2$ and $4 \times 4$ matrices over $\Z[t^{\pm 1}]$ that, since~2026, is amenable to being solved using AI.
Given the rate of progress of AI models and the worrying trend of theorems not being posted to the arXiv, I have come to the conclusion that the present outcome 
(a publicly available human write-up, posted to the arXiv,  with a clear AI disclosure, and a digestion of the proof) is preferable to the following alternative: in the near future, 
 a passerby feeds the statement of Theorem~\ref{thm:ZTori} directly into a more advanced (possibly non-publicly available) AI model,
an answer is generated and is then posted without any further edits somewhere that is difficult to access.

\subsection*{Organisation}
The remainder of this paper is devoted to the proof of Theorem~\ref{thm:main}.
Section~\ref{sec:cocycle} reformulates its statement in the form of an equation of the form $\sigma(Q)=SQ$ where the matrix~$S$ and the involution $\sigma$ are respectively defined in Notation~\ref{not:MatrixS} and~\ref{not:Involutionsigma}.
Section~\ref{sec:Local} solves local versions of this equation, whereas Section~\ref{sec:global} then obtains a global solution from the local ones.
The short Appendix~\ref{sec:fixed} collects some properties of the norm subring of $\Z[t^{\pm 1}].$

\subsection*{Acknowledgments}
The author was partially supported by the NSF grant DMS~2303674.

\subsection*{Conventions}
Given a ring~$R$ with involution, we call a map~$f \colon M \to N$ of $R$-modules \emph{antilinear} if it is additive and satisfies~$f(rm)=\overline{r}f(m)$ for every~$r \in R$ and every~$m \in M$.

\section{Set-up and reformulation}
\label{sec:cocycle}

The goal of this section is to reduce Theorem~\ref{thm:main} to an equality of matrices that is (anti)linear in the columns of the unknown matrix (Theorem~\ref{thm:splitting}).
To begin with, we obtain a normal form for nondegenerate hermitian matrices that present the same module as $H$.

\begin{proposition}
\label{prop:A1=0Automatic}
If a nondegenerate size~$2$ hermitian matrix~$A$ over~$\Lambda:=\Z[t^{\pm 1}]$ presents the same module as~$H$,  then~$A$ must be of the form
\begin{equation}
\label{eq:normal-form}
 A = \begin{pmatrix} c(t-1)(t^{-1}-1) & a(t-1) \\ \overline{a}(t^{-1}-1) & b(t-1)(t^{-1}-1) \end{pmatrix} 
 \end{equation}
for some~$a, b, c \in \Lambda$ with~$b = \overline{b},c = \overline{c}, a(1)=\pm 1$ and~$a\overline{a}-bc(t-1)(t^{-1}-1)=1$.
\end{proposition}
\begin{proof}
Isomorphic finitely presented~$\Lambda$-modules have equal elementary ideals (these are also often called \emph{Fitting ideals}). 
Since the module~$M(H)$ is presented by~$H$, its~$0$-th elementary ideal is~$\langle \det(H) \rangle = \langle -(t-1)(t^{-1}-1) \rangle$ and its first elementary ideal is generated by the~$1 \times 1$ minors (the entries of~$H$), and is thus~$\langle 0, t-1, t^{-1}-1 \rangle = \langle t-1 \rangle$.

If~$A = \bsm \gamma & \alpha \\ \overline{\alpha} & \beta \esm$ is hermitian (which implies~$\gamma = \overline{\gamma}$ and~$\beta = \overline{\beta}$), then~$M(A) \cong M(H)$ implies equality of elementary ideals,  so~$\det(A)=\pm \det(H)=\varepsilon (t-1)(t^{-1}-1)$ with $\varepsilon=\pm 1$ and 
\[ \langle \gamma, \alpha, \overline{\alpha}, \beta \rangle = \langle t-1 \rangle. \]
This equality of ideals requires that every entry of~$A$ must be a multiple of~$t-1$. 
Factor~$\alpha \in \langle t-1 \rangle$ as~$\alpha = a(t-1)$ for some~$a \in \Lambda$. 
For the diagonal entries,~$\gamma \in \langle t-1 \rangle$ and~$\gamma=\overline{\gamma}$ are seen to imply that~$\gamma = c(t-1)(t^{-1}-1)$ for some symmetric polynomial~$c = \overline{c}$.
By the exact same reasoning,~$\beta \in \langle t-1 \rangle$ and~$\beta = \overline{\beta}$ forces~$\beta = b(t-1)(t^{-1}-1)$ for some symmetric polynomial~$b = \overline{b}$.

Combining these calculations,  we deduce that
$$ \varepsilon (t-1)(t^{-1}-1)=\det(A)=-(t-1)(t^{-1}-1)(a\ol a-(t-1)(t^{-1}-1)bc).$$
Dividing by~$(t-1)(t^{-1}-1)$ and setting $t=1$ yields $\varepsilon=-a(1)^2$ from which it follows that $\varepsilon=-1$.
This shows that $a(1)=\pm 1$ and~$a\overline{a}-bc(t-1)(t^{-1}-1)=1$, concluding the proof of the proposition.
\end{proof}

\begin{remark}
The condition~$a(1)=\pm1$ is in fact implied by~$a\ol a-bc(t-1)(t^{-1}-1)=1$.
Indeed, evaluating the latter at~$t=1$ and using $a(1)=\ol a(1)$, gives
$a(1)^2=1$.  
We nevertheless record the condition~$a(1)=\pm1$ since it will come up frequently in the sequel.
\end{remark}

As announced above, we now recast the equation $\overline{P}^THP=A$ into an equation that is (anti)linear in the columns of the unknown matrix.
This requires some notation.

\begin{notation}
\label{not:MatrixS}
For brevity, we set
\[
 x:=t-1,\qquad
 u:=x\overline{x}=2-t-t^{-1}.
\]
Next, consider the matrices
\[
 S:=\begin{pmatrix}\ol a&bx\\c\overline{x}&a\end{pmatrix},\qquad
 D:=\begin{pmatrix}t&0\\0&1\end{pmatrix}.
\]
A direct calculation gives
\[
 HS=
 \begin{pmatrix}0&x\\\overline{x}&0\end{pmatrix}
 \begin{pmatrix}\ol a&bx\\c\overline{x}&a\end{pmatrix}
 =\begin{pmatrix}cx\overline{x}&ax\\\ol a \overline{x}&bx\overline{x}\end{pmatrix}=A.
\]
Also, the final identity in Proposition~\ref{prop:A1=0Automatic} gives
\begin{equation}\label{eq:HS}
\det(S)=a\ol a-bcx\overline{x}=1.
\end{equation}
\end{notation}

\begin{notation}
\label{not:Involutionsigma}
Consider the following involution on the set~$M_2(\La)$ of $2\times 2$ matrices over $\Lambda$:
$$
 \sigma(X):=D\ol X D^{-1}.
$$
We record the basic properties of~$\sigma$, all of which will be used
repeatedly.  
The map~$\sigma$ is a ring automorphism of~$M_2(\La)$: it is
additive, multiplicative and unital. 
 In particular, for every invertible matrix~$X\in GL_2(\Lambda)$, we have
 $$\sigma(X^{-1})=\sigma(X)^{-1}.$$
Since~$b=\ol b$ and~$c=\ol c$,  a direct calculation using~$t\overline{x}=-x$ and~$t^{-1}x=-\overline{x}$ gives
\[
 \sigma(S)=D\ol S D^{-1}
 =\begin{pmatrix}a&t\,b \overline{x}\\t^{-1}c x&\ol a\end{pmatrix}
 =\begin{pmatrix}a&-bx\\-c\overline{x}&\ol a\end{pmatrix}.
\]
Since the last matrix is the adjugate of~$S$ and~$\det(S)=1$, we obtain
\begin{equation}\label{eq:cocycle}
 \sigma(S)=S^{-1}.
\end{equation}
\end{notation}

As we explain below,  Theorem~\ref{thm:main} is a consequence of the following result.

\begin{theorem}\label{thm:splitting}
When~$a(1)=1$, there is a matrix~$Q\in\SL_2(\La)$ such that
\begin{equation}\label{eq:split}
 \sigma(Q)=SQ.
\end{equation}
\end{theorem}

We show how Theorem~\ref{thm:splitting} implies Theorem \ref{thm:main} from the introduction.

\begin{proof}[Proof of Theorem \ref{thm:main} assuming Theorem~\ref{thm:splitting}]
We first assume that~$a(1)=1$.
Let~$Q\in\SL_2(\La)$ be as in Theorem \ref{thm:splitting}, and put
$P:=Q^{-1}$. 
 From~$\sigma(Q)=SQ$ and the multiplicativity of~$\sigma$, we obtain
 \[
 \sigma(P)=\sigma(Q)^{-1}=(SQ)^{-1}=Q^{-1}S^{-1}=PS^{-1};
\]
 As a consequence, we deduce that
\begin{equation*}
  S=\sigma(P)^{-1}P.
\end{equation*}
We claim that~$\overline{P}^TH=H\sigma(P)^{-1}$.
Note the factorisation~$H=xJ$ with $J:=\bsm 0&1\\-t^{-1}&0\esm.$
For every matrix~$P=\bsm p&q\\r&s\esm\in
\SL_2(\La)$, we have~$\sigma(P)=\bsm
 \ol p&t\ol q\\t^{-1}\ol r&\ol s
 \esm$ and~$\det(\sigma(P))=\ol{\det (P)}=1$, so that
\[
\sigma(P)^{-1}
= \begin{pmatrix}\ol s&-t\ol q\\-t^{-1}\ol r&\ol p\end{pmatrix}
 =\begin{pmatrix}0&-t\\1&0\end{pmatrix}
 \begin{pmatrix}\ol p&\ol r\\\ol q&\ol s\end{pmatrix}
 \begin{pmatrix}0&1\\-t^{-1}&0\end{pmatrix}
 =J^{-1}\ol P^{\,T}J.
\]
It follows that~$\ol P^{\,T}J=J\sigma(P)^{-1}$ and multiplying by~$x$ yields the claim.

The claim now promptly implies the theorem when $a(1)=1$:
\begin{equation*}
 \ol P^{\,T}HP
 =H\sigma(P)^{-1}P
 =H S=A.
\end{equation*}
Next, assume that~$a(1)=-1$.
Consider the diagonal matrix~$\Delta:=\operatorname{diag}(1,-1)$.
The Hermitian matrix~$A':=\ol \Delta^{T}A\Delta=\Delta A \Delta$ has~$b'=b,c'=c$ and~$a'=-a$.
In particular $a'(1)=1$, so the preceding argument gives a~$P'\in\SL_2(\La)$ with~$\overline{P'}^{\,T}A' P'=H$.
Setting~$P:=\Delta P'$, now gives~$\overline{P}^{T}AP=\overline{P'}^{\,T}\Delta A \Delta P'=\overline{P'}^{\,T}A' P'=H.$
This concludes the proof of the theorem.
\end{proof}

The following sections are devoted to the proof of Theorem~\ref{thm:splitting}.
The idea is to first prove local versions of the theorem (Section~\ref{sec:Local}) and to then deduce the result from the local ones (Section~\ref{sec:global}).

\begin{convention}
\label{conv:aone}
From now on we assume that the matrix~$A$ from~\eqref{eq:normal-form} satisfies
\begin{equation}\label{eq:aone}
 a(1)=1.
\end{equation}
\end{convention}

\section{The local version of Theorem~\ref{thm:splitting}.}
\label{sec:Local}

In the remainder of this article,  we set~$R:=\Z[u]$ and, given a maximal ideal~$\mathfrak{m} \subset R$, we write~$R_{\mathfrak{m}}$ for the localisation at~$\mathfrak{m}$, and set $\Lambda_{\mathfrak{m}}:=\Lambda \otimes_R R_{\mathfrak{m}}$.
Thus~$\Lambda_{\mathfrak m}$ is obtained from~$\Lambda$ by inverting the elements of~$R\setminus\mathfrak m$,  \emph{not} the elements of~$\Lambda \setminus \mathfrak{m}.$
The goal of this section is to prove the following local version of Theorem~\ref{thm:splitting}.

\begin{theorem}\label{thm:local}
For every maximal ideal~$\mathfrak m\subset R$, there is a matrix
$Q_{\mathfrak m}\in \GL_2(\La_{\mathfrak m})$ satisfying
$$\sigma(Q_{\mathfrak m})=SQ_{\mathfrak m}.$$
\end{theorem}

Recall the notation $x:=(t-1),u:=x\ol x=2-t-t^{-1}$ and set 
$$f:=(1+t)(1+t^{-1})=4-u.$$
\begin{remark}
\label{rem:Casework}
Theorem~\ref{thm:local} will be proved by considering four cases.
First case:~$u\in \mathfrak{m}$ and~$2 \notin \mathfrak{m}$,  
second case:~$f\in \mathfrak{m}$ and~$2 \notin \mathfrak{m}$,  
third case:~$\mathfrak{m}=(2,u)$,  
and fourth case~$uf \notin \mathfrak{m}$.
We argue that these four cases are exhaustive.
Since~$uf \notin \mathfrak{m}$ is case 4,  this reduces to showing that~$uf  \in \mathfrak{m}$ leads to the first three cases.
When~$uf  \in \mathfrak{m}$,  since~$\mathfrak{m}$ is prime,~$u \in \mathfrak{m}$ or~$f \in \mathfrak{m}$.
Assuming that~$2 \notin \mathfrak{m}$ leads to cases 1 and 2.
If~$2\in\mathfrak m$, the congruence~$f=4-u\equiv-u\pmod 2$ shows
that these two alternatives coincide and force
$2,u\in\mathfrak m$ hence~$\mathfrak m\supseteq(2,u)$; since
$R/(2,u)=\mathbb F_2$ is a field,~$(2,u)$ is itself maximal and
$\mathfrak m=(2,u)$; this is case~$3$.
\end{remark}

In each of these cases,  we 
produce a matrix~$Q \in \GL_2(\La_{\mathfrak{m}})$ with~$\sigma(Q)=SQ$.
Before beginning the casework, we record a fact about the ring~$\Lambda_{\mathfrak{m}}$ for later use.

\begin{lemma}
\label{lem:locality}
Let~$\mathfrak m\subset R$ be a maximal ideal.
If either (a) $f\in\mathfrak m$ and~$2\notin\mathfrak m$ or (b)~$\mathfrak m=(2,u)$, then the ring~$\La_{\mathfrak m}$ is local, with residue field isomorphic to~$\kappa:=R_{\mathfrak m}/\mathfrak mR_{\mathfrak m}$.
\begin{enumerate}
\item[\textup{(a)}] If~$f\in\mathfrak m$ and~$2\notin\mathfrak m$, then the quotient map~$\La_{\mathfrak m}\to\kappa$ sends $t$ to $-1$ and restricts to the quotient map $R_{\mathfrak m}\to\kappa$.
\item[\textup{(b)}] If~$\mathfrak m=(2,u)$, then the quotient map~$\La_{\mathfrak m}\to\kappa \cong \F_2$ sends $t$ to $1$ and restricts to the quotient map $R_{\mathfrak m}\to\kappa$; in particular it reduces integer coefficients modulo $2$.
\end{enumerate}
For every~$z\in\La_{\mathfrak m}$, the elements~$z$ and~$\ol z$ have the same image in~$\kappa$.
Finally, an element of~$\La_{\mathfrak m}$ is a unit
if and only if its image in the residue field is nonzero.
\end{lemma}
\begin{proof}
We begin by describing the quotient~$\La_{\mathfrak m}/\mathfrak m\La_{\mathfrak m}$ and show that it is local.
We will then deduce that $\La_{\mathfrak m}$ is itself local and finally consider the properties of the quotient map.
Lemma \ref{lem:eigenspaces} shows that~$\La_{\mathfrak m}$ is a
free~$R_{\mathfrak m}$-module with basis~$1,t$.
More precisely,  the relation~$ t^2-(2-u)t+1=0$
implies that $\Lambda \cong R[t]/(t^2-(2-u)t+1)$ and thus~$\Lambda_{\mathfrak m} \cong R_{\mathfrak m}[t]/(t^2-(2-u)t+1)$.
Reducing modulo~$\mathfrak{m}R_{\mathfrak{m}}$ and using~$T\in \kappa[T]$ (resp.~$\bar u\in\kappa$) to denote the image of~$t \in \Lambda_{\mathfrak{m}}$ (resp. the image of~$u \in \Lambda_{\mathfrak{m}}$) under the quotient map,  we obtain
\[
 \La_{\mathfrak m}/\mathfrak m\La_{\mathfrak m}
 \;\cong\;\kappa[T]/\bigl(T^2-(2-\bar u)T+1\bigr).
\]
Note that $T$ is invertible in the quotient because the relation gives
$T\bigl((2-\bar u)-T\bigr)=1$.

In case (a),~$4-u=f\in\mathfrak m$ gives~$\bar u=4$, so the relation
becomes~$0=T^2+2T+1=(T+1)^2$. 
 In case (b),~$\kappa=\Z[u]/(2,u) \cong \mathbb F_2$ and~$\bar u=0$, so this time, the relation becomes~$0=T^2+1=(T-1)^2$.
Set~$s:=T+1$ in case (a) and~$s:=T-1$ in case (b) so
that, in both cases, the quotient is
\[
 \La_{\mathfrak m}/\mathfrak m\La_{\mathfrak m} \cong  \kappa[s]/(s^2)=\kappa\oplus\kappa s.
\]
An element~$c_0+c_1s$ with~$c_0\neq0$ is a unit (with inverse~$c_0^{-1}-c_0^{-2}c_1s$) and the elements with~$c_0=0$ are not units 
and form the ideal $(s)$.  
Thus the nonunits of the quotient coincide with the ideal~$(s)$.
It follows that~$ \La_{\mathfrak m}/\mathfrak m\La_{\mathfrak m} $ is a local ring with maximal ideal~$(s)$ and residue field isomorphic to~$\kappa$.
\begin{claim}
\label{claim:QuotientLocal}
The ring~$\La_{\mathfrak m}$ is local.
\end{claim}
\begin{proof}[Proof of Claim~\ref{claim:QuotientLocal}]
Use~$\pi \colon \La_{\mathfrak m} \to \La_{\mathfrak m}/\mathfrak{m}\La_{\mathfrak m}=:\ol \La_{\mathfrak m}$ to denote the quotient map.
The map~$\mathfrak{a} \mapsto \pi(\mathfrak{a})$ induces a bijection between the set of ideals of~$\La_{\mathfrak m}$ containing~$\mathfrak{m}\Lambda_{\mathfrak{m}}$ and the set of ideals of~$\ol \La_{\mathfrak m}$.
Since~$\Lambda_{\mathfrak{m}}/\mathfrak{a} \cong \ol \La_{\mathfrak m}/\pi(\mathfrak{a})$, an ideal~$\mathfrak{a} \subset \Lambda_{\mathfrak{m}}$ containing $\mathfrak{m}\Lambda_{\mathfrak{m}}$ is maximal if and only if~$\pi(\mathfrak{a})\subset \ol \Lambda_{\mathfrak{m}}$ is maximal.

The ring~$\ol \La_{\mathfrak m}=\La_{\mathfrak m}/\mathfrak{m}\La_{\mathfrak m}$ is local with unique maximal ideal $(s)$.
The claim therefore reduces to proving that every maximal ideal of~$\Lambda_{\mathfrak{m}}$ contains~$\mathfrak{m}\Lambda_{\mathfrak{m}}$.
Indeed, this would then imply that~$\Lambda_{\mathfrak{m}}$ has a single maximal ideal, namely~$\mathfrak{n}_0:=\pi^{-1}((s))$,  and that~$\Lambda_{\mathfrak{m}}/\mathfrak{n}_0 \cong \ol \Lambda_{\mathfrak{m}}/(s) \cong \kappa.$
In particular, this  would show that~$\Lambda_{\mathfrak{m}}$ is local.

We now show that every maximal ideal~$\mathfrak{n} \subset \Lambda_{\mathfrak{m}}$ contains~$\mathfrak{m}\Lambda_{\mathfrak{m}}$.
Consider the subring~$R_{\mathfrak{m}} \subset \Lambda_{\mathfrak{m}}$ and note that~$\Lambda_{\mathfrak{m}}$ is integral over~$R_{\mathfrak{m}}$:
indeed it is a finitely generated free module over the
local ring~$R_{\mathfrak m}$ with basis~$1,t$; see e.g.~\cite[Proposition 5.1]{AtiyahMacdonald}.
By~\cite[Corollary 5.8]{AtiyahMacdonald},  since~$\Lambda_{\mathfrak{m}}$ is integral over~$R_{\mathfrak{m}}$ and since~$\mathfrak n \subset \Lambda_{\mathfrak{m}}$ is maximal, so is~$\mathfrak{n} \cap R_{\mathfrak{m}} \subset R_{\mathfrak m}$.
Since~$R_{\mathfrak m}$ is local with maximal ideal~$\mathfrak{m}R_{\mathfrak m}$,  
 it follows that~$\mathfrak{n} \cap R_{\mathfrak{m}}=\mathfrak{m}R_{\mathfrak m}$.
In particular,~$\mathfrak{n} \supset \mathfrak{m}R_{\mathfrak{m}}$ and since~$\mathfrak{n}$ is an ideal of~$\Lambda_{\mathfrak{m}}$,  it contains the ideal generated by that set, i.e.~$\mathfrak{n} \supset (\mathfrak{m}R_{\mathfrak{m}})\Lambda_{\mathfrak{m}}=\mathfrak{m}\Lambda_{\mathfrak{m}}.$

As explained above, this proves that~$\Lambda_{\mathfrak{m}}$ is local and thus concludes the proof of Claim~\ref{claim:QuotientLocal}.
\end{proof}
Next, we verify the properties of the quotient map
$$\La_{\mathfrak m}\to \La_{\mathfrak m}/ \mathfrak m \La_{\mathfrak m} \cong \kappa \oplus \kappa s \to  \kappa. $$
The composition~$\La_{\mathfrak m}\to
\La_{\mathfrak m}/\mathfrak m\La_{\mathfrak m}\to\kappa$ kills~$s=T \mp 1$,
because it sends~$t$ to~$-1$ in case (a) and to~$1$ in case (b).
By construction,  it restricts to the canonical map on~$R_{\mathfrak m}$. 
 It is the unique
ring homomorphism with these properties because~$1,t$ generate
$\La_{\mathfrak m}$ over~$R_{\mathfrak m}$. 
 For~$z=\alpha+\beta t \in \Lambda_{\mathfrak{m}}$ with~$\alpha,\beta\in R_{\mathfrak m}$, one has
$\ol z=\alpha+\beta t^{-1}$,  so~$z$ and~$\ol z$ have the same residue.
Finally, in any local ring the nonunits are exactly the elements of the
maximal ideal, which is the kernel of the residue map; this proves
the last assertion of the lemma.
\end{proof}

The next four sections are devoted to the casework needed to prove Theorem~\ref{thm:local}.

\subsection{First case:~$u\in\mathfrak m$ and~$2\notin\mathfrak m$.}
\label{sub:1}

Recall that $R:=\Z[u]$.
For the first two cases of Theorem~\ref{thm:local}, the matrix~$Q_{\mathfrak{m}}$ will be constructed using the following lemma.
\begin{lemma}
\label{lem:IplusS}
Let $R'$ be a localisation of $R$ and set~$\La':=\La\otimes_RR'$.
If a matrix~$\Sigma\in M_2(\La')$ satisfies~$\sigma(\Sigma)=\Sigma^{-1}$ and~$I+\Sigma$ is invertible over~$\La'$, then~$Q:=(I+\Sigma)^{-1}$
satisfies 
$$\sigma(Q)=\Sigma Q.$$
\end{lemma}
\begin{proof}
Since~$\sigma$ is a ring homomorphism,
$\sigma(I+\Sigma)=I+\sigma(\Sigma)=I+\Sigma^{-1}$, and therefore
\[
 \sigma(I+\Sigma)\,\Sigma=(I+\Sigma^{-1})\Sigma=\Sigma+I.
\]
Thus~$\sigma(I+\Sigma)=(I+\Sigma)\Sigma^{-1}$.  Inverting both sides and
using~$\sigma(X^{-1})=\sigma(X)^{-1}$ yields
\[
\sigma(Q)
=\sigma\bigl((I+\Sigma)^{-1}\bigr)
 =\sigma(I+\Sigma)^{-1}
 =\Sigma\,(I+\Sigma)^{-1}
 =\Sigma Q.
\]
This concludes the proof of the lemma.
\end{proof}

The following proposition proves the first case of Theorem~\ref{thm:local}.

\begin{proposition}\label{prop:case2}
If~$\mathfrak m\subset R$ is a maximal ideal with~$u\in\mathfrak m$ and~$2\notin\mathfrak m$, then there is an invertible matrix~$Q_{\mathfrak m}\in GL_2(\La_{\mathfrak m})$ satisfying
$$\sigma(Q_{\mathfrak m})=SQ_{\mathfrak m}.$$
\end{proposition}
\begin{proof}
We show that~$I+S$ is invertible over~$\La_{\mathfrak m}$ and then apply
Lemma \ref{lem:IplusS} with~$R'=R_{\mathfrak{m}}$ and~$\Sigma=S$; indeed recall from~\eqref{eq:cocycle} that $\sigma(S)=S^{-1}$.
Using the definition of~$S$, we see that
\[
 \det(I+S)
= \det \begin{pmatrix} 1+\ol a&bx\\c\overline{x}&1+a\end{pmatrix}
=1+a+\overline{a}+\overbrace{\det(S)}^{=1}
 =2+a+\ol a.
\]
The element~$a+\ol a$ is fixed by the involution and therefore lies in
$R=\mathbb Z[u]$ by Lemma \ref{lem:eigenspaces}.
Write~$a+\ol a=g(u)$ for some~$g\in\mathbb Z[u]$.  
Evaluation at~$t=1$ is a ring
homomorphism~$\La\to\mathbb Z$ which sends~$u=2-t-t^{-1}$ to~$0$ and
sends both~$a$ and~$\ol a$ to~$a(1)=1$; recall Convention~\ref{conv:aone}.
It follows that~$g(0)=(a+\ol a)(1)=2a(1)=2$  and therefore
$g(u)-2\in u\,\mathbb Z[u]\subseteq\mathfrak m$ (because~$u \in \mathfrak{m}$), so
\[
 \det(I+S)=2+g(u)\equiv 4\mod{\mathfrak m}.
\]
Since~$2\notin\mathfrak m$ and~$\mathfrak m$ is prime,  it follows that~$4\notin\mathfrak m$. 
As a consequence,~$\det(I+S)$ is a unit
 of~$R_{\mathfrak m}$ and hence of~$\La_{\mathfrak m}$.  
Thus~$I+S$ is
invertible in~$\Lambda_{\mathfrak{m}}$.
Lemma \ref{lem:IplusS} applied to~$\Sigma=S$ gives the result.
\end{proof}

\subsection{Second case:~$f\in\mathfrak m$ and~$2\notin\mathfrak m$.}
\label{sub:2}
We establish the second case of Theorem~\ref{thm:local}.

\begin{proposition}\label{prop:case3}
For every maximal ideal~$\mathfrak m\subset R$ with~$f\in\mathfrak m$ and~$2\notin\mathfrak m$, there is an invertible matrix~$Q_{\mathfrak m}\in GL_2(\La_{\mathfrak m})$ satisfying
$$\sigma(Q_{\mathfrak m})=SQ_{\mathfrak m}.$$
\end{proposition}
\begin{proof}
We first set up some notation.
Since~$f\in\mathfrak m$ and~$2\notin\mathfrak m$,  Lemma \ref{lem:locality}(a) implies that~$\La_{\mathfrak m}$ is a local ring.
This lemma also states that the residue field of $\Lambda_\mathfrak{m}$ is~$\kappa=R_{\mathfrak m}/\mathfrak mR_{\mathfrak m}$,
that the residue map sends~$t\mapsto-1$, and that for every~$z \in \Lambda_\mathfrak{m}$, the elements~$z$ and $\ol z$ have equal
residues.  
Writing~$a_-\in\kappa$ for the residue of~$a$, the conjugate~$\ol a$ also has residue~$a_-$, and thus the residue of~$\det(I+S)=2+a+\ol a$ is
$2(1+a_-)$.

We first prove the proposition when~$a_-\neq-1$.
In this case,~$2(1+a_-)\neq0$ in~$\kappa$ (the characteristic of
$\kappa$ is not~$2$, since~$2\notin\mathfrak m$),
so by the last sentence
of Lemma \ref{lem:locality}, the element~$\det(I+S)$ is a unit of~$\La_{\mathfrak m}$, and Lemma \ref{lem:IplusS} with~$\Sigma=S$ shows that~$Q_{\mathfrak{m}}:=(I+S)^{-1}$ satisfies $\sigma(Q_{\mathfrak{m}})=SQ_{\mathfrak{m}}$,  exactly as in the first case.

We now assume that~$a_-=-1$.  
Set~$W:=\bsm 0&-1\\1&0\esm$.
The argument consists of checking that the same proof as above works with~$S^W:=\sigma(W)^{-1}SW$ substituted for~$S$.
By Lemma \ref{lem:transfer} below, the parameter of~$S^W$
playing the role of~$a$ is~$a':=t^{-1}\ol a$.  
The residue of~$t^{-1}$
is~$(-1)^{-1}=-1$ 
and therefore
the residue of~$a'$ is
\[
 a'_-
 =(-1)\cdot a_-
 =-a_-
 =1.
\]
We can now repeat the computation from the proof of Proposition~\ref{prop:case2} with~$S^W$ instead of $S$.
First, since~$\det(S^W)=1$ (see Lemma~\ref{lem:transfer} for the details),  
we calculate
$$\det(I+S^W)=1+a'+\ol{a'}+\det(S^W)=2+a'+\ol{a'}.$$
Taking the residue yields~$2(1+a_-')\neq0$ in~$\kappa$,
so by the last sentence
of Lemma \ref{lem:locality} the element~$\det(I+S^W)$ is a unit of~$\La_{\mathfrak m}$.
It follows that $I+S^W$ is invertible over $\Lambda_{\mathfrak{m}}$.
Since we also have~$\sigma(S^W)=(S^W)^{-1}$ (the calculation is performed in Lemma~\ref{lem:transfer} below), Lemma~\ref{lem:IplusS} applies with
$\Sigma=S^W$ and states that~$Q'=(I+S^W)^{-1}$ satisfies~$\sigma(Q')=S^WQ'$.
Finally, Lemma \ref{lem:transfer} below shows that~$Q_{\mathfrak{m}}:=WQ'$ satisfies~$\sigma(Q_{\mathfrak{m}})=SQ_{\mathfrak{m}}$.
\end{proof}

We establish the properties of the matrix~$S^W$ that were used during the proof of the previous proposition.
We collect these in the form of a lemma since these calculations will be used again in the third case of Theorem~\ref{thm:local}.

\begin{lemma}
\label{lem:transfer}
Let $R'$ be a localisation of $R$,  set~$\La':=\La\otimes_RR'$,  and
\[
 W:=\begin{pmatrix}0&-1\\1&0\end{pmatrix},
 \qquad S^W:=\sigma(W)^{-1}SW\in M_2(\La').
\]
The matrix~$S^W$ has the same shape and properties as~$S$, namely
$$
 S^W
 =\begin{pmatrix}ta&cx\\b\overline{x}&t^{-1}\ol a\end{pmatrix}
 =:\begin{pmatrix}
 \ol{a'}&b'x\\c'\overline{x}&a'
 \end{pmatrix},
$$
where these~$a',b',c'\in \Lambda$ satisfy~$ b'=\ol{b'}, c'=\ol{c'}$ and $a'\ol{a'}-b'c'x\overline{x}=1, a'(1)=1$, as well as
$$\sigma(S^W)=(S^W)^{-1} \quad \text{and} \quad \det(S^W)=1.$$
Finally, if~$Q'\in\GL_2(\La')$ satisfies~$\sigma(Q')=S^WQ'$, then~$Q:=WQ'$ satisfies~$\det (Q)=\det(Q')$ and
$$\sigma(Q)=SQ.$$
\end{lemma}
\begin{proof}
We first verify the formulas for $S^W$ and its coefficients. 
Since~$W$ has integer entries,~$\ol W=W$ and~$\sigma(W)=DWD^{-1}=\bsm 0&-t\\t^{-1}&0\esm$.
This has~$\det(\sigma (W))=1$ and~$\sigma(W)^{-1}=\bsm0&t\\-t^{-1}&0\esm$.
Next, a direct calculation shows that~$SW=\bsm \ol a&bx\\c\overline{x}&a\esm
 \bsm0&-1\\1&0\esm =\bsm bx&-\ol a\\a&-c\overline{x}\esm$.
Using~$-t\overline{x}=x$ and~$-t^{-1}x=\overline{x}$, we deduce that
\[
 S^W=\sigma(W)^{-1}(SW)
 =\begin{pmatrix}0&t\\-t^{-1}&0\end{pmatrix}
 \begin{pmatrix}bx&-\ol a\\a&-c\overline{x}\end{pmatrix}
 =\begin{pmatrix}ta&-tc\overline{x}\\-t^{-1}bx&t^{-1}\ol a\end{pmatrix}
 =\begin{pmatrix}ta&cx\\b\overline{x}&t^{-1}\ol a\end{pmatrix}.
\]
The parameters~$b',c'$ are fixed by the involution because~$b$
and~$c$ are.  
Proposition~\ref{prop:A1=0Automatic} gives
\[
 a'\ol{a'}-b'c'x\overline{x}=(t^{-1}\ol a)(ta)-cb\,x\overline{x}=a\ol a-bcx\overline{x}=1.
\]
Similarly,  evaluation at~$t=1$ 
 gives~$a'(1)=a(1)=1$; recall Convention~\ref{conv:aone}.

We now verify the further properties of $S^W$.
Since~$\sigma$ is a ring automorphism with~$\sigma^2=\mathrm{id}$,
we have~$\sigma(\sigma(W)^{-1})=W^{-1}$.
This equality together with~$\sigma(S)=S^{-1}$ implies that
\[
 \sigma(S^W)=W^{-1}\,\sigma(S)\,\sigma(W)
 =W^{-1}S^{-1}\sigma(W)
 =\bigl(\sigma(W)^{-1}SW\bigr)^{-1}=(S^W)^{-1}.
\]
A direct calculation also shows that~$\det(S^W)=\det(\sigma(W)^{-1})\det(S)\det(W)=1$.

It only remains to verify the final sentence of the lemma.
To do so,  assume that~$Q'\in\GL_2(\La')$ satisfies~$\sigma(Q')=S^WQ'$, and set~$Q:=WQ'$.
The multiplicativity of~$\sigma$ gives
\[
 \sigma(Q)=\sigma(WQ')=\sigma(W)\sigma(Q')=\sigma(W)S^WQ'
 =\sigma(W)\,\sigma(W)^{-1}SW\,Q'=S\,(WQ')=SQ.
\]
Finally,~$\det(Q)=\det(WQ')=\det (W)\det (Q')=\det (Q')$ because~$\det (W)=1$.
\end{proof}

\subsection{Third case:~$\mathfrak m=(2,u)$.}
\label{sub:3}
For~$\mathfrak m=(2,u)$, Lemma \ref{lem:locality}(b) asserts that~$\La_{\mathfrak m}$ is a local ring with
 residue field~$\kappa \cong \mathbb F_2$,  
that the residue map sends~$t\mapsto1$ and
reduces integer coefficients modulo~$2$, 
and that~$z$,~$\ol z$ always have
equal residues.  
Since the argument from Lemma~\ref{lem:IplusS} is not
available (the residue of~$\det(I+S)=2+a+\ol a$ is
$0+1+1=0$ in~$\mathbb F_2$), we instead build the matrix~$Q_{\mathfrak{m}}$
column by column.
The key lemma we will use to carry this out (namely Lemma~\ref{lem:completion}) requires we introduce some notation.

\begin{notation}
\label{not:taui}
Throughout this section,~$R'$ will denote a localisation of~$R$ (the case~$R'=R$
is allowed) and~$\La':=\La\otimes_RR'$.
Given a matrix~$\Sigma\in\SL_2(\La')$ satisfying~$
 \sigma(\Sigma)=\Sigma^{-1}$,  we define antilinear involutions on~$(\La')^2$ by
\[
 \tau_1^{\Sigma}(v):=t^{-1}\Sigma^{-1}D\ol v,
  \qquad
   \tau_2^{\Sigma}(v):=\Sigma^{-1}D\ol v.
\]
We spell out why $\tau_2^\Sigma$  is an involution; the argument for $\tau_1^\Sigma$ is analogous. 
Inverting both sides of~$\sigma(\Sigma)=D\ol\Sigma D^{-1}=\Sigma^{-1}$ gives~$D\ol{\Sigma^{-1}}\ol D=\Sigma$, and thus
\[
 (\tau_2^{\Sigma})^2(v)
 =\Sigma^{-1}D\,\ol{\Sigma^{-1}D\ol v}
 =\Sigma^{-1}\bigl(D\ol{\Sigma^{-1}}\ol D\bigr)v=\Sigma^{-1}\Sigma v=v.
\]
The motivation for defining these maps is that for a size $2$ matrix~$Q=(q_1\ q_2)$, reading off columns shows that~$\sigma(Q)=\Sigma Q$ (i.e.~$D\ol Q=\Sigma QD$) if and only if~$\tau_1^\Sigma(q_1)=q_1$ and~$\tau_2^\Sigma(q_2)=q_2$. 
This can be seen by multiplying the definitions of $\tau_1^\Sigma(v)$ and $\tau_2^\Sigma(v)$ by~$t\Sigma$ and~$\Sigma$ respectively to obtain 
\begin{equation}\label{eq:tauMotivation}
\tau_1^\Sigma(v)=v \Longleftrightarrow D\ol v=t\Sigma v,
 \qquad
\tau_2^\Sigma(w)=w\Longleftrightarrow D\ol w=\Sigma w.
\end{equation}
As mentioned in the introduction, this will be relevant both during the proof of Theorem~\ref{thm:local} when~$\mathfrak m=(2,u)$ and during the proof of Theorem~\ref{thm:splitting}, when working over $R$.
\end{notation}

In order to prove Theorem~\ref{thm:local} when~$\mathfrak m=(2,u)$, we require the following lemma.
\begin{lemma}
\label{lem:completion}
Suppose~$\Sigma\in\SL_2(\La')$ satisfies~$ \sigma(\Sigma)=\Sigma^{-1}$. 
 If a unimodular vector~$v\in(\La')^2$ satisfies~$\tau_1^{\Sigma}(v)=v$, then there is a vector
$w'\in(\La')^2$ for which~$Q:=(v\ \ w')$ satisfies
\[
 Q \in\SL_2(\La')
 \quad\text{and}\quad
 \sigma(Q)=\Sigma Q.
\]
\end{lemma}
Here,  unimodular means that the two coordinates of~$v$ generate~$\La'$. 
We postpone the proof of this lemma to the end of this section and instead focus on proving Theorem~\ref{thm:local} when~$\mathfrak m=(2,u)$.

\begin{proposition}\label{prop:case4}
For the maximal ideal~$\mathfrak m=(2,u) \subset R$, there is a~$Q_{\mathfrak m}\in GL_2(\La_{\mathfrak m})$ satisfying
$$\sigma(Q_{\mathfrak m})=SQ_{\mathfrak m}.$$
\end{proposition}
\begin{proof}
 The argument has three steps.
Firstly we construct a vector~$v$ fixed by~$\tau_1^{S}$.
Secondly we show that either~$v$ or an analogous vector attached to~$S^W$ is unimodular.
Thirdly we invoke Lemma~\ref{lem:completion} to construct~$Q_{\mathfrak{m}}$.

We construct a vector~$v$ fixed by~$\tau_1^{S}$.
Using our convention that~$a(1)=1$,  the Laurent polynomial~$a-1$ vanishes at~$t=1$.  
It follows that~$a-1$ is divisible by~$x=t-1$ in~$\La$.
In particular, we can factor~$t^N(a-1)=(t-1)g$ for some~$g\in\mathbb Z[t]$ and some~$N \geq 0$, and we write
\[
 a=1+xd,\qquad d:=t^{-N}g\in\La.
\]
 Applying the involution, we obtain~$\ol a=1+\overline{x}\ol d$ and,  since~$\overline{x}=-t^{-1}x$, this can be rewritten as~$(1-\ol a)/x=t^{-1}\ol d\in\La.$
Now define
\[
 v:=\begin{pmatrix}b\\(1-\ol a)/x\end{pmatrix}
   =\begin{pmatrix}b\\t^{-1}\ol d\end{pmatrix}.
\]
We verify that~$v$ is indeed fixed by $\tau_1^S$.
\begin{claim}
\label{claim:Fixedv}
The vector $v$ satisfies~$\tau_1^{S}(v)=v. $
\end{claim}
\begin{proof}
We must show that~$t^{-1}S^{-1}D\ol v=v$.
Multiplying by~$S$ shows that this is equivalent to~$t^{-1}D\ol v=Sv$, and we check the latter coordinate by
coordinate. 
Using~$D=\operatorname{diag}(t,1)$ and~$\ol{(t^{-1}\ol d)}=td$ as well as~$\ol b=b$,  the left hand side is
\[
 t^{-1}D\ol v=t^{-1}\begin{pmatrix}t\ol b\\td\end{pmatrix}
 =\begin{pmatrix}b\\d\end{pmatrix}.
\]
Recalling that~$S=\bsm \ol a&bx\\c\overline{x}&a\esm,$ for the right hand side, the first coordinate of $Sv$ is, as required, 
\[
 \ol a\,b+bx \frac{1-\ol a}{x}=\ol a\,b+b(1-\ol a)=b.
\]
For the second coordinate,  we must show~$bc\overline{x}+a\,(1-\ol a)/x=d$.  Multiply the left side by~$x$ and use the determinant identity~$bcx\overline{x}=a\ol a-1$ from \eqref{eq:HS}:
\begin{align*}
 x\left(bc\overline{x}+a\,\frac{1-\ol a}{x}\right)
 &=bcx\overline{x}+a(1-\ol a)\\
 &=(a\ol a-1)+a-a\ol a=a-1=xd.
\end{align*}
The ring~$\La_{\mathfrak m}$ is a domain (it is a localisation of the
domain~$\La$ at a multiplicative set of nonzero elements), so the
nonzero factor~$x$ can be cancelled, proving the second coordinate
identity.  
Hence~$\tau_1^{S}(v)=v$,  establishing Claim~\ref{claim:Fixedv}.
\end{proof}

This concludes the first part of the proof.
We now move on to the second and third parts of the proof.
Namely,  we show that either~$v$ or the analogous vector obtained from~$S^W$ is unimodular.
We then conclude using Lemma~\ref{lem:completion} (either to $S$ or to $S^W$) to construct~$Q_{\mathfrak{m}}$.

We first assume that~$b$ or~$d$ is a unit of $\Lambda_{\mathfrak{m}}$.
Over a local ring,  a vector is unimodular if and only if at least one of
its coordinates is a unit 
(the vector is not unimodular if and only if the ideal generated by its coordinates is contained in the maximal ideal if and only if each coordinate lies in the maximal ideal if and only if each coordinate is a nonunit).
Thus, if~$b$ or $d$ is a unit of~$\Lambda_{\mathfrak m}$ (in the latter case so is~$t^{-1}\ol d$), 
then~$v$ is unimodular.
 Lemma \ref{lem:completion} with~$R'=R_{\mathfrak{m}}$ and~$\Sigma=S$ then produces the required matrix~$Q_{\mathfrak{m}}\in\SL_2(\La_{\mathfrak m})$ with~$\sigma(Q_{\mathfrak{m}})=SQ_{\mathfrak{m}}$.

Next, we assume that~$b$ and~$d$ are not units of~$\Lambda_{\mathfrak{m}}$.
 By Lemma~\ref{lem:transfer}, the matrix~$S^W$ has the same shape as~$S$,
with the parameters~$a':=t^{-1}\ol a,b':=c,c':=b$ satisfying the same
conditions as~$a,b,c$, including~$a'(1)=1$.  
The construction of the previous paragraphs therefore applies verbatim to~$S^W$: writing~$a'=1+xd'$ with~$d'\in\La$,
the vector
\[
 v':=\begin{pmatrix}b'\\(1-\ol{a'})/x\end{pmatrix}
   =\begin{pmatrix}c\\t^{-1}\ol{d'}\end{pmatrix}
\]
satisfies~$\tau_1^{S^W}(v')=v'$, by the identical computation with
$(a,b,c,d)$ replaced by~$(a',b',c',d')$.
Indeed the computation only uses the shape of the matrix $S^W$, the conditions of Lemma \ref{lem:transfer}, and the factorization~$a'=1+xd'$.

It remains to show that~$v'$ is unimodular. 
If~$c$ is a unit of~$\Lambda_{\mathfrak m}$,  the first coordinate of~$v'$ is a unit and we are done.  
We can therefore assume~$b,c$ and~$d$ are all nonunits of~$\La_{\mathfrak{m}}$.
We assert that then~$d'$ is a unit of~$\La_{\mathfrak{m}}$ so that the second coordinate of~$v'$ is a unit.  
Using~$a'=t^{-1}\ol a$ and~$\ol a=1+\overline{x}\ol d=1-t^{-1}x\ol d$,  we have
\[
 a'-1
 =t^{-1}(1-t^{-1}x\ol d)-1
 =t^{-1}-1-t^{-2}x\ol d
 =\overline{x}-t^{-2}x\ol d
 =-t^{-1}x-t^{-2}x\ol d.
\]
Dividing by~$x$,  we obtain
\[
 d'=\frac{a'-1}{x}=-t^{-1}-t^{-2}\ol d.
\]
Now we compute residues in~$\mathbb F_2$ using Lemma
\ref{lem:locality}(b).
The residue of~$t$ (hence of~$t^{-1}$ and
$t^{-2}$) is~$1$; the residue of~$d$ is~$0$ because~$d$ is a nonunit of
the local ring~$\La_{\mathfrak m}$; and the residue of~$\ol d$ equals
that of~$d$, hence is also~$0$.  
Therefore the residue of~$d'$ is~$-1-0=1\neq0$,
so~$d'$ is a unit of~$\La_{\mathfrak m}$ , as asserted.

We have therefore proved that~$v'$ is unimodular.
 Lemma \ref{lem:completion} now applies with~$\Sigma=S^W$ (its hypotheses, $\sigma(S^W)=(S^W)^{-1}$ and~$\det(S^W)=1$, hold by Lemma
\ref{lem:transfer}), producing
$Q'\in\SL_2(\La_{\mathfrak m})$ with~$\sigma(Q')=S^WQ'$. 
The final sentence of Lemma
\ref{lem:transfer} then shows that the matrix~$Q_{\mathfrak m}:=WQ'$ satisfies~$\sigma(Q_{\mathfrak m})=SQ_{\mathfrak m}$ and
$\det (Q_{\mathfrak m})=\det (Q')=1$.
\end{proof}

We now establish Lemma~\ref{lem:completion} whose proof was postponed.
Recall that this lemma states ``Suppose~$\Sigma\in\SL_2(\La')$ satisfies~$ \sigma(\Sigma)=\Sigma^{-1}$.  If a unimodular vector
$v\in(\La')^2$ satisfies~$\tau_1^{\Sigma}(v)=v$, then there is a vector
$w'\in(\La')^2$ for which~$
 Q:=(v\ \ w')$ satisfies $Q\in\SL_2(\La')~$ and~$\sigma(Q)=\Sigma Q.$'' 

\begin{proof}[Proof of Lemma~\ref{lem:completion}]
We first fix some notation.
Recall from Notation~\ref{not:taui} that for $i=1,2$, the involutions $\tau_i \colon (\Lambda')^2 \to (\Lambda')^2$ are defined as~$\tau_1^{\Sigma}(v):=t^{-1}\Sigma^{-1}D\ol v$ and~$\tau_2^{\Sigma}(v):=\Sigma^{-1}D\ol v$.
We abbreviate~$\tau_i=\tau_i^{\Sigma}$ for the rest
of this proof.

We begin with a first attempt at completing $v\in (\Lambda')^2$ into the required matrix $Q \in SL_2(\Lambda')$.
Since~$v=(v_1,v_2)^T $ is unimodular,  there are~$\alpha,\beta \in \Lambda'$ with
$\alpha v_1+\beta v_2=1$.
Set~$w:=(-\beta,\alpha)^T$.
Since~$\det(v,w)=\alpha v_1+\beta v_2=1$, this~$w$ completes~$v$ to an invertible matrix, but there is no reason for~$w$ to satisfy the equation~$\tau_2(w)=w$ that the second column of~$Q$ must satisfy for the equation~$\sigma(Q)=\Sigma Q$ to hold; recall Notation~\ref{not:taui}.
The remaining steps correct this initial~$w$.

\begin{claim}
There exists a $w' \in (\Lambda')^2$ that satisfies~$\det(v,w')=1$ and~$
\tau_2(w')=w'$.
\end{claim}
\begin{proof}
We first calculate the defect $\tau_2(w)-w \in (\Lambda')^2$.
Since~$\det(v,w)=1$, the pair~$(v,w)$ is a basis of~$(\La')^2$.
Write~$\tau_2(w)-w=gv+hw$ in this basis for some~$g,h \in \Lambda'$.
Taking the determinant with~$v$
gives~$\det(v,\tau_2(w)-w)=\det(v,gv+hw)=h\det(v,w)=h$.
We show that $\det(v,\tau_2(w)-w)=0$ by proving that~$\det(v,\tau_2(w))=1$.
This follows by using~$v=\tau_1(v)=t^{-1}\Sigma^{-1}D\ol v$ and calculating:
\begin{align*}
 \det(v,\tau_2(w))
 &=\det\bigl(t^{-1}\Sigma^{-1}D\ol v,\Sigma^{-1}D\ol w\bigr)\\
 &=t^{-1}\det(\Sigma^{-1})\det(D)\,\ol{\det(v,w)}
 =t^{-1}\cdot1\cdot t\cdot 1=1.
\end{align*}
 We have proved that $h=\det(v,\tau_2(w)-w)=0$ and so we obtain that for some~$g\in\La'$,
\[
 \tau_2(w)-w=gv.
\]
We now obtain a constraint on~$g$.
Apply~$\tau_2$ to the above equation,  and use the antilinearity of~$\tau_2$ together with~$\tau_2^2=1$ and~$\tau_2(v)=tv$ to obtain
\[
 -gv=w-\tau_2(w)=\tau_2(\tau_2(w)-w)=\tau_2(gv)=\ol g\,\tau_2(v)
 =t\ol g\,v.
\]
Thus~$(g+t\ol g)v=0$.  
This is an equality of vectors; multiplying its
first coordinate by~$\alpha$,  its second coordinate by~$\beta$, and
adding gives~$(g+t\ol g)(\alpha v_1+\beta v_2)=g+t\ol g=0\in \Lambda'$.  
The fourth formula of Lemma~\ref{lem:eigenspaces} (in its localised form) now gives, for some~$r\in R'$,
$$g=xr.$$
As announced, we now conclude the proof by modifying the vector~$w \in (\Lambda')^2$ in order to obtain a new~$w'$ with $\det(v,w')=1$ but that now satisfies~$\tau_2(w')=w'$.
Put~$k:=tr$. 
 Then~$\ol k=t^{-1}r$, and
\[
 k-t\ol k=tr-r=(t-1)r=xr=g.
\]
For~$w':=w+kv$,  the antilinearity of $\tau_2$ and the equality~$\tau_2(v)=tv$ give
\[
 \tau_2(w')
 =\tau_2(w)+\ol k\,\tau_2(v)
 =(w+gv)+t\ol k\,v
 =w+(g+t\ol k)v
 =w+kv
 =w'.
\]
Moreover~$\det(v,w')=\det(v,w+kv)=\det(v,w)=1$, concluding the proof of the claim.
\end{proof}
As explained in Notation~\ref{not:taui}, since the columns of~$Q:=(v\ w')$ satisfy $\tau_1(v)=v$ and $\tau_2(w')=w'$,  the matrix $Q$ satisfies~$\sigma(Q)=\Sigma Q$.
This concludes the proof of the lemma.
\end{proof}

\subsection{Fourth case:~$uf\notin\mathfrak m$.}
\label{sub:4}

The next result establishes the fourth and final case of Theorem~\ref{thm:local}.

\begin{proposition}\label{prop:case1}
For every maximal ideal~$\mathfrak m\subset R$ with~$uf\notin\mathfrak m$,  there is a~$Q_{\mathfrak m}\in GL_2(\La_{\mathfrak m})$ satisfying
$$\sigma(Q_{\mathfrak m})=SQ_{\mathfrak m}.$$
\end{proposition}
\begin{proof}
Recall the antilinear involutions~$\tau_1^{S}$ and~$\tau_2^{S}$ on~$\La_{\mathfrak m}^2$ from Notation~\ref{not:taui}.
Since~$uf\notin\mathfrak m$,  Lemma~\ref{lem:descent} below shows that the fixed module~$E_2^{\mathfrak{m}}:=\{w\in \La_{\mathfrak m}^2:\tau_2^{S}(w)=w\}$ is free of rank two over~$R_{\mathfrak m}$,
and that any~$R_{\mathfrak m}$-basis~$m_1,m_2$ of~$E_2^{\mathfrak{m}}$ is a~$\La_{\mathfrak m}$-basis of~$\La_{\mathfrak m}^2$.  
Here we applied this lemma to~$T=\tau_2^{S}$ and $n=2$.
In what follows we set
$$
h_+:=1+t \quad \text{and} \quad h_-:=1+t^{-1}.
$$
We assert that multiplication by~$h_-$ carries~$E_2^{\mathfrak{m}}$ into the fixed module of~$\tau_1^{S}$.  
Indeed, for~$w\in E_2^{\mathfrak{m}}$ we have
$\tau_1^{S}(w)=t^{-1}\tau_2^{S}(w)=t^{-1}w$, and thus, by
antilinearity,~$\ol{h_-}=h_+$, and $t^{-1}h_+=t^{-1}+1=h_-,$ we obtain~$
 \tau_1^{S}(h_-w)=\ol{h_-}\,\tau_1^{S}(w)=h_+t^{-1}w=h_-w,$ as asserted.

Next,  given a basis $m_1,m_2$ of $E_2^{\mathfrak{m}}$, we consider the matrix
\[
 Q_{\mathfrak m}:=(h_-m_1\ \ m_2)=(m_1\ \ m_2)
 \begin{pmatrix}h_-&0\\0&1\end{pmatrix} \in M_2(\La_{\mathfrak m}).
\]
Since the columns~$q_1, q_2$ of~$Q_{\mathfrak m}$ are fixed by~$\tau_1^S$ and~$\tau_2^S$ respectively (here we used the assertion),  it follows that~$Q_{\mathfrak m}$ satisfies~$\sigma(Q_{\mathfrak m})=SQ_{\mathfrak m}$; recall Notation~\ref{not:taui}.
It only remains to argue that~$Q_{\mathfrak m}$ is invertible over~$\La_{\mathfrak m}$.
We do so by noting that both of its factors are invertible over~$\La_{\mathfrak m}$.
For the first, this is because~$m_1,m_2$ is a~$\La_{\mathfrak m}$-basis of~$\La_{\mathfrak m}^2$.
For the second, this reduces to showing that~$h_-$ is a unit of~$\La_{\mathfrak m}$.
For this, note that the element~$f:=h_+h_-$ divides the unit~$uf$ of~$R_{\mathfrak m}$, so~$f\in R_{\mathfrak m}^{\times}$ and hence~$h_-\in \La_{\mathfrak m}^{\times}$, as required.
We have therefore constructed the required matrix~$ Q_{\mathfrak m}$ and this concludes the proof of the proposition.
\end{proof}

We now prove the lemma that was used during the proof of this proposition.

\begin{lemma}\label{lem:descent}
Fix a maximal ideal~$\mathfrak m\subset R$ with~$uf\notin\mathfrak m$.
Let~$T \colon \La_{\mathfrak m}^n\to \La_{\mathfrak m}^n$ be a~$\La_{\mathfrak m}$-antilinear involution and set
\[
 N:=\{v\in \La_{\mathfrak m}^n:T(v)=v\}.
\]
Then~$N$ is a free~$R_{\mathfrak m}$-module of rank~$n$, and every~$R_{\mathfrak m}$-basis of~$N$ is
a~$\La_{\mathfrak m}$-basis of~$\La_{\mathfrak m}^n$.
\end{lemma}
\begin{proof}
First, note that $N$ is an~$R_{\mathfrak m}$-module because the elements of~$R_{\mathfrak{m}}$ are fixed by the involution~$- \colon \Lambda_{\mathfrak{m}} \to \Lambda_{\mathfrak{m}}.$
We prove that~$N$ is free over~$R_{\mathfrak m}$.
The element~$u \in R_{\mathfrak m}$ divides the unit~$uf \in R_{\mathfrak m}$ in~$R_\mathfrak{m}$,
so~$u\in R_{\mathfrak m}^{\times}$. 
Set~$\lambda:=-u^{-1}x\in \La_{\mathfrak m}$ and~$\lambda+\ol\lambda=-u^{-1}(x+\ol x)=-u^{-1}(t+t^{-1}-2)=1.$
Consider the map
$$P \colon \La_{\mathfrak m}^n \to \La_{\mathfrak m}^n,  \quad v \mapsto \ol\lambda\,v+\lambda\,T(v).$$
This map is~$R_{\mathfrak m}$-linear because the involution fixes~$R_{\mathfrak m}$.
It takes values in~$N$ because~$T(P(v))=\lambda\,T(v)+\ol\lambda\,v=P(v)$.
Furthermore $P$ restricts to the identity on~$N$ because~$P(w)=(\ol\lambda+\lambda)w=w$ for all~$w\in N$.
Hence~$P$ is a projection of~$\La_{\mathfrak m}^n$ onto~$N$, and so
$\La_{\mathfrak m}^n=N\oplus\ker P$ as~$R_{\mathfrak m}$-modules.
It follows that~$N$ is a direct summand of
$\La_{\mathfrak m}^n\cong R_{\mathfrak m}^{2n}$
 and is therefore a finitely generated projective~$R_{\mathfrak m}$-module.  
Since~$R_{\mathfrak m}$ is local,  Kaplansky's theorem implies that $N$ is free over $R_{\mathfrak m}$~\cite{Kaplansky}.

We now prove that an~$R_{\mathfrak m}$-basis of~$N$ is a~$\La_{\mathfrak m}$-basis of~$\La_{\mathfrak m}^n$.
\begin{claim}
The $R_{\mathfrak m}$-module~$N$ spans~$\La_{\mathfrak m}^n$ over~$\La_{\mathfrak m}$.
\end{claim}
\begin{proof}
For~$v\in \La_{\mathfrak m}^n$,  both~$p:=v+T(v)$ and~$q:=tv+t^{-1}T(v)$ are fixed by~$T$:
\[
 T(p)=T(v)+v=p,
 \qquad
 T(q)=t^{-1}T(v)+tv=q.
\]
Moreover~$q-t^{-1}p=(t-t^{-1})v$ so,  setting~$\theta:=t-t^{-1}$, we obtain the relation~$q-t^{-1}p=\theta v$.
The calculation~$
 \theta^2
 =t^2-2+t^{-2}
 =(t+t^{-1})^2-4
 =(2-u)^2-4
 =-uf \notin \mathfrak{m}$
implies that $\theta^2 \in \La_{\mathfrak m}$ is a unit.
It follows that~$\theta \in \La_{\mathfrak m}$ is a unit.
We deduce that 
\begin{equation}\label{eq:span}
 v=\theta^{-1}\bigl(q-t^{-1}p\bigr).
\end{equation}
This implies that every vector of~$\La_{\mathfrak m}^n$ is a~$\La_{\mathfrak m}$-linear combination of fixed
vectors, as claimed.
\end{proof}
\begin{claim}
If~$w_1,w_2\in N$ satisfy~$w_1+tw_2=0$,  then $w_1=0=w_2$.
\end{claim}
\begin{proof}
Applying~$T$ gives~$w_1+t^{-1}w_2=0$.
Subtracting~$w_1+t^{-1}w_2=0$ from~$w_1+tw_2=0$ yields~$\theta w_2=~0$.
Since~$\theta$ is a unit of~$\La_{\mathfrak m}$, it follows that~$w_2=0$ and so
$w_1=0$, as claimed.
\end{proof}

We conclude the proof that an~$R_{\mathfrak m}$-basis of~$N$ is a~$\La_{\mathfrak m}$-basis of~$\La_{\mathfrak m}^n$.
Let~$n_1,\dots,n_k$ be an~$R_{\mathfrak m}$-basis of~$N$.  
It spans~$\La_{\mathfrak m}^n$ over~$\La_{\mathfrak m}$ by
the first claim. 
 If~$\sum_iz_in_i=0$, write~$z_i=\alpha_i+\beta_it \in \La_{\mathfrak m}$ with
$\alpha_i,\beta_i\in R_{\mathfrak m}$,
then~$w_1:=\sum_i\alpha_in_i$ and
$w_2:=\sum_i\beta_in_i$ lie in~$N$ and satisfy~$w_1+tw_2=0$, so
$w_1=w_2=0$ by the second claim, and~$R_{\mathfrak m}$-independence of the~$n_i$ gives
$\alpha_i=\beta_i=0$, that is,~$z_i=0$.  
Thus~$n_1,\dots,n_k$ is a~$\La_{\mathfrak m}$-basis, and so~$\La_{\mathfrak m}^n$ is free of rank~$k$ over~$\La_{\mathfrak m}$.  
Since every nonzero commutative ring has the invariant basis number property, $k=n$.
\end{proof}

\subsection{Conclusion of the proof of Theorem~\ref{thm:local}.}

Recall the statement that we wish to prove: for every maximal ideal~$\mathfrak m\subset R$, there is a matrix~$Q_{\mathfrak m}\in \GL_2(\La_{\mathfrak m})$ satisfying~$\sigma(Q_{\mathfrak m})=SQ_{\mathfrak m}.$

\begin{proof}[Proof of Theorem~\ref{thm:local}]
The four cases from Propositions~\ref{prop:case2},\ref{prop:case3},\ref{prop:case4}, and~\ref{prop:case1} produce an invertible matrix~$Q$ over~$\La_{\mathfrak m}$ with
$\sigma(Q)=SQ$.
These cases are exhaustive by Remark~\ref{rem:Casework}.
\end{proof}

\section{Conclusion of the proof: the local to global step}
\label{sec:global}

This section assembles the local solutions of Theorem~\ref{thm:local}
into a global one, proving Theorem~\ref{thm:splitting}.
We begin by introducing some notation and discussing the strategy of the proof.

\begin{notation}
Recall the maps~$\tau_1(v):=t^{-1}S^{-1}D\ol v$ and
$\tau_2(w):=S^{-1}D\ol w$ on~$\La^2$ from Notation~\ref{not:taui}.
Define the~$R$-modules
\[
 E_1=\{v\in\La^2:\tau_1(v)=v\},
 \qquad 
 E_2=\{w\in\La^2:\tau_2(w)=w\}.
\]
These are $R$-modules because $\tau_i$ is additive and antilinear and because $R$ is fixed by the involution.

Multiplying the defining equations by~$tS$ and by~$S$ respectively,
gives 
\begin{equation}\label{eq:Econditions}
 v\in E_1\Longleftrightarrow D\ol v=tSv,
 \qquad
 w\in E_2\Longleftrightarrow D\ol w=Sw.
\end{equation}
The motivation for these maps is that for a size $2$ matrix~$Q=(q_1\ q_2)$, reading off columns shows that~$\sigma(Q)=SQ$ (i.e.~$D\ol Q=SQD$) if and only if~$q_1\in E_1$ and~$q_2\in E_2$. 
\end{notation}

We now discuss the strategy underlying the proof of Theorem~\ref{thm:splitting}.
As noted above,  our goal is to find~$q_1\in E_1$ and~$q_2\in E_2$ with~$\det(q_1,q_2)=1$.
A necessary condition for this to occur is that the $E_i$ be nontrivial. 
This will be established by showing that the~$E_i$ are~$R$-free of rank~$2$ (Proposition~\ref{prop:Free}).
Choosing nontrivial elements~$q_1' \in E_1,q_2'\in E_2$ leads to a nontrivial (but possibly singular) matrix~$Q'=(q_1'\ q_2')$ with~$\sigma(Q')=SQ'$.
The challenge is to force~$\det(Q')=1$; this is where it will matter that the $E_i$ are $R$-free of rank $2$ instead of merely being nontrivial.
In order to achieve this, we construct a basis~$e_1,e_2 \in E_1$ and an isomorphism~$\alpha \colon E_2 \to \alpha(E_2) \subset E_1$ (Construction~\ref{cons:Alpha}) with~$e_2 \in \im(\alpha)$ (Proposition~\ref{prop:ImageAlpha}): taking~$q_1:=e_1$ and~$q_2:=\varepsilon \alpha^{-1}(e_2)$ for some~$\varepsilon \in \{ \pm 1\}$ will then yield the additional condition on the determinant of the matrix~$Q=(q_1 \ q_2)$.

\begin{remark}
\label{rem:Eifp}
We collect some remarks on the $R$-modules $E_1$ and $E_2$.
\begin{enumerate}
\item  The $R$-module~$E_i$ is finitely presented for $i=1,2$:
the ring~$R=\mathbb Z[u]$ is Noetherian,~$\La^2$ is free of rank $4$ over~$R$,
and~$E_i$ is the kernel of the~$R$-linear map~$\La^2 \to \La^2, v\mapsto v-\tau_i(v)$ 
\item Because localisation is exact, it commutes with the kernels defining
the~$E_i$; thus, for every maximal ideal~$\mathfrak m \subset R$, the localisation~$(E_i)_{\mathfrak m}$
is the set of solutions of the same
equations~\eqref{eq:Econditions} inside~$(\La_{\mathfrak m})^2$.
\end{enumerate}
\end{remark}

We now work towards the first step of the proof of Theorem~\ref{thm:splitting}, namely showing that~$E_1$ and~$E_2$ are~$R$-free.
This will be done by constructing isomorphisms~$(E_1)_{\mathfrak{m}} \cong (F_1)_{\mathfrak{m}}$ and~$(E_2)_{\mathfrak{m}} \cong (F_2)_{\mathfrak{m}}$ for every maximal ideal $\mathfrak{m} \subset R$, where~$F_1$ and~$F_2$ are free. 

\begin{construction}
\label{cons:Fi}
Alongside~$E_1$ and~$E_2$, consider the $R$-modules of solutions of the corresponding equations in which the matrix~$S$ is replaced by the
identity matrix:
\[
 F_1=\{v\in\La^2:D\ol v=tv\},
 \qquad
 F_2=\{w\in\La^2:D\ol w=w\}.
\]
We argue that the~$F_i$ are~$R$-free of rank~$2$.
For an element~$v=(v_1,v_2) \in F_1$,  since~$D\ol v=(t\ol v_1,\ol v_2)$ and~$tv=(tv_1,tv_2)$, the equation~$D\ol v=tv$ says~$\ol v_1=v_1$ and~$\ol v_2=tv_2$.
This means that~$v_1 \in R$ and~$v_2 \in (1+t^{-1})R$; see Lemma \ref{lem:eigenspaces}.
 Similarly,  for~$w=(w_1,w_2) \in F_2$,  the equation~$D\ol w=w$ says~$t\ol w_1=w_1$ and~$\ol w_2=w_2$, i.e.~$w_1\in (1+t)R$ and~$w_2\in R$.
Setting~$h_+:=1+t$ and~$h_-=1+t^{-1}$ and using~$e_1,e_2$ to denote the canonical basis of~$\La^2$, we conclude that
\[
 F_1=Re_1\oplus h_-Re_2,
 \qquad
 F_2=h_+Re_1\oplus Re_2.
\]
In particular each~$F_i$ is a free~$R$-module of rank two, as claimed.
\end{construction}

We now establish the first step in the proof of Theorem~\ref{thm:splitting}.

\begin{proposition}
\label{prop:Free}
The~$R$-modules~$E_1$ and~$E_2$ are free of rank~$2$:
\begin{equation}\label{eq:Efree}
 E_1\cong R^2,
 \qquad E_2\cong R^2.
\end{equation}
In addition,  given a maximal ideal $\mathfrak{m} \subset R$ and using~$Q_{\mathfrak m}=(q_1^{\mathfrak{m}},q_2^{\mathfrak{m}})\in \GL_2(\La_{\mathfrak m})$ to denote a local solution from Theorem~\ref{thm:local}, an~$R_{\mathfrak m}$-basis of~$(E_1)_{\mathfrak m}$ and~$(E_2)_{\mathfrak m}$ is given by 
$$
(E_1)_{\mathfrak m} = R_{\mathfrak m} q_1^{\mathfrak{m}} \oplus R_{\mathfrak m} \;h_-q_2^{\mathfrak{m}};
\qquad
(E_2)_{\mathfrak m} = R_{\mathfrak m}h_+q_1^{\mathfrak{m}} \oplus R_{\mathfrak m}q_2^{\mathfrak{m}}.
\qquad
$$
\end{proposition}
\begin{proof}
We first prove the local version of the statement.
Recall the basis of~$F_i \cong R^2$ from Construction~\ref{cons:Fi} for $i=1,2$.
By the localisation clause of Lemma \ref{lem:eigenspaces},
the same descriptions hold after localisation: for every maximal
ideal~$\mathfrak m \subset R$, 
\[
 (F_1)_{\mathfrak m}=R_{\mathfrak m}e_1\oplus h_-R_{\mathfrak m}e_2,
 \qquad
 (F_2)_{\mathfrak m}=h_+R_{\mathfrak m}e_1\oplus R_{\mathfrak m}e_2,
\]
and these are the respective sets of solutions of the equations~$D\ol v=tv$ and
$D\ol w=w$ inside~$(\La_{\mathfrak m})^2$.
The matrix~$Q_{\mathfrak m}\in \GL_2(\La_{\mathfrak m})$ from
Theorem~\ref{thm:local} satisfies~$\sigma(Q_{\mathfrak m})
=SQ_{\mathfrak m}$ which can be rewritten as~$D\ol{Q_{\mathfrak m}}=SQ_{\mathfrak m}D$.
A verification using this equality shows that for $i=1,2$ multiplication by~$Q_{\mathfrak m}$ induces well-defined~$R_{\mathfrak m}$-linear isomorphisms
\[
 \Phi_i \colon \;(F_i)_{\mathfrak m}\longrightarrow(E_i)_{\mathfrak m},
 \qquad \Phi_i(z)=Q_{\mathfrak m}z.
\]
This implies that~$(E_i)_{\mathfrak m}\cong(F_i)_{\mathfrak m}\cong R_{\mathfrak m}^2$ is~$R_{\mathfrak m}$-free of rank~$2$ for~$i=1,2$.
Since~$E_i$ is finitely presented (recall Remark~\ref{rem:Eifp}) and free at every maximal localisation,  we deduce that~$E_i$ is a finitely generated
projective~$R$-module; see e.g.~\cite[Chapitre II, §5, no. 2, Th\'eor\`eme 1]{Bourbaki} or~\cite[Lemma~10.78.2 (Tag
00NX)]{Stacks}.
It follows that each~$E_i$ is a free~$R$-module, see e.g.~\cite[Chapter~V, Corollary~4.12]{Lam}.
Localising at any maximal ideal and comparing with the rank
computed above gives~$E_1\cong R^2$ and~$E_2\cong R^2$, as required.

We conclude by proving the statement about bases.  
Let~$q_1^{\mathfrak{m}},q_2^{\mathfrak{m}}$ denote the columns of~$Q_{\mathfrak m}$.  
The isomorphism~$\Phi_1$ carries the basis~$e_1,h_-e_2$ of~$(F_1)_{\mathfrak m}$ to~$Q_{\mathfrak m}e_1=q_1^{\mathfrak{m}};Q_{\mathfrak m}(h_-e_2)=h_-q_2^{\mathfrak{m}}.$
Thus~$q_1^{\mathfrak{m}},\;h_-q_2^{\mathfrak{m}}$ is an~$R_{\mathfrak m}$-basis of~$(E_1)_{\mathfrak m}$.  
Likewise~$\Phi_2$ carries the basis
$h_+e_1,e_2$ of~$(F_2)_{\mathfrak m}$ to~$h_+q_1^{\mathfrak{m}},q_2^{\mathfrak{m}}$, which is therefore an~$R_{\mathfrak m}$-basis of~$(E_2)_{\mathfrak m}$.
This proves the last clause of the proposition.
\end{proof}

We move on to the second step of the proof of Theorem~\ref{thm:splitting}.
Recall that now that we have proved that the $E_i$ are free, we can construct a matrix~$Q'$ that satisfies~$\sigma(Q')=SQ'$; the remaining challenge is to arrange the condition that $\det(Q')=1$.
As mentioned above, we do this by constructing a basis~$e_1,e_2 \in E_1$ and an isomorphism~$\alpha \colon E_2 \to \alpha(E_2) \subset E_1$ with $e_2 \in \im(\alpha)$: taking~$q_1=e_1$ and~$q_2=\varepsilon \alpha^{-1}(e_2)$ for some~$\varepsilon \in \{ \pm 1\}$ will then yield the additional condition on the determinant of the matrix~$Q=(q_1 \ q_2)$.

\begin{construction}(The map~$\alpha$.)
\label{cons:Alpha}
Multiplication by~$h_-:=1+t^{-1}$ defines an~$R$-linear map
\[
 \alpha \colon E_2\longrightarrow E_1,
 \qquad \alpha(w):=h_-w.
\]
It is well defined because, for~$w\in E_2$,  the antilinearity of~$\tau_1$ and the equality~$\tau_2(w)=w$ give
\[
 \tau_1(h_-w)
 =\ol{h_-}\,\tau_1(w)
  =\ol{h_-} t^{-1}\tau_2(w)
 =h_+ t^{-1}\tau_2(w)
 =t^{-1}h_+w
 =h_-w.
\]
We calculate a matrix for the localised map~$\alpha_{\mathfrak m} \colon (E_2)_{\mathfrak{m}} \to (E_1)_{\mathfrak{m}}$ with respect to the bases from
Proposition~\ref{prop:Free}.
 Fix a maximal ideal~$\mathfrak m \subset R$, and recall that~$h_+q_1,\;q_2$ is a basis
of~$(E_2)_{\mathfrak m}$ while~$q_1,\;h_-q_2$ is a basis of
$(E_1)_{\mathfrak m}$.  
On these bases vectors, $\alpha$ is given by
\[
 \alpha(h_+q_1)=h_-h_+q_1=fq_1,
 \qquad
 \alpha(q_2)=h_-q_2.
\]
Thus with respect to these bases, the localised map
$\alpha_{\mathfrak m}$ is represented by the matrix
$\operatorname{diag}(f,1)$.
\end{construction}

As outlined above,  the next step of our strategy is to understand the image of~$\alpha$.
As an intermediate step, we describe the cokernel of~$\alpha$.

\begin{lemma}
\label{lem:Cokernel}
As an $R$-module, the cokernel of~$\alpha \colon E_2 \to E_1$ satisfies
$$ \coker(\alpha) \cong R/(f).$$
\end{lemma}
\begin{proof}
Set~$L:=\coker(\alpha)=E_1/\alpha(E_2)$ for brevity; this is a finitely generated
$R$-module.  
Localisation is exact and thus commutes with cokernels.
Given a maximal ideal~$\mathfrak{m} \subset R$, it follows that~$L_{\mathfrak m} \cong \coker(\alpha_{\mathfrak m})$, as~$R_{\mathfrak{m}}$-modules.
Since~$\alpha_{\mathfrak m} \colon (E_2)_{\mathfrak{m}} \to  (E_1)_{\mathfrak{m}}$ is represented by the matrix~$\operatorname{diag}(f,1)$,  its cokernel is
$$ L_{\mathfrak m} \cong \coker(\alpha_{\mathfrak{m}}) \cong R_{\mathfrak m}/fR_{\mathfrak m}\oplus R_{\mathfrak m}/R_{\mathfrak m}
=R_{\mathfrak m}/fR_{\mathfrak m}.$$
 If~$f\notin\mathfrak m$, then~$f$
is a unit of~$R_{\mathfrak m}$ and this vanishes, whence
\begin{equation}
\label{eq:Lm}
 L_{\mathfrak m}\cong
 \begin{cases}
 0&f\notin\mathfrak m\\
 R_{\mathfrak m}/fR_{\mathfrak m},&f\in\mathfrak m.
 \end{cases}
\end{equation}
In particular~$(fL)_{\mathfrak m}=fL_{\mathfrak m}=0$ for every
maximal ideal $\mathfrak{m} \subset R$, and therefore~$fL=0$.  
Thus~$L$ is a module over~$R/(f)$, finitely
generated because~$L$ is finitely generated over~$R$.
Here, note that
\[
 R/(f)=\mathbb Z[u]/(4-u)\cong\mathbb Z.
\]
We now show that~$L\cong R/(f)$ as $R/(f)\cong \Z$-modules.

As usual, we first prove the local version.
Namely, we show that~$L\cong R/(f)$ after localising at every maximal ideal of~$R/(f)$.
The maximal ideals of~$R/(f)$ correspond bijectively to the maximal ideals
$\mathfrak m \subset R$ containing~$f$.
For such an ideal~$\mathfrak m \subset R$,  using~\eqref{eq:Lm}, we have the isomorphisms
$$
L \otimes_{R/(f)} (R/(f))_{\mathfrak m/(f)}
\cong L \otimes_{R/(f)} (R/(f) \otimes_R R_{\mathfrak{m}})
\cong L\otimes_R R_{\mathfrak m}
= L_{\mathfrak m}
\cong R_{\mathfrak m}/fR_{\mathfrak m}
\cong (R/(f))_{\mathfrak m/(f)}.
$$
So~$L$ is a finitely generated module over~$\mathbb Z\cong R/(f)$ (hence finitely presented,~$\mathbb Z$ being Noetherian) that is free of rank one at every maximal localisation.
As in the proof of Proposition~\ref{prop:Free},  this implies that $L$ is finitely generated projective over the ring~$\mathbb Z\cong R/(f)$ and therefore is free.  
Its rank is one, as one sees by localising at any maximal ideal. 
 Therefore~$L\cong R/(f)$,  thus concluding the proof of the lemma.
\end{proof}

Next, we describe the image of $\alpha \colon E_2 \to E_1$.

\begin{proposition}
\label{prop:ImageAlpha}
There is an $R$-basis~$e_1,e_2 \in E_1 \cong R^2$ such that 
$$ \im(\alpha)= fRe_1\oplus Re_2.$$
\end{proposition}
\begin{proof}
Use Lemma~\ref{lem:Cokernel} to fix an $R$-linear isomorphism~$L\cong R/(f)$ and let~$\pi \colon E_1\twoheadrightarrow R/(f) \cong \Z$ be the composition of the quotient map~$E_1\to L=E_1/\alpha(E_2)$ with this isomorphism.  
It follows that
\begin{equation}\label{eq:kerpi}
\ker(\pi)=\im(\alpha).
\end{equation}
Recall from Proposition~\ref{prop:Free} that~$E_i \cong R^2$.
Choose an $R$-basis~$e'_1,e'_2$ of~$E_1 \cong R^2$ and consider the integers~$m:=\pi(e'_1),n:=\pi(e'_2) \in R/(f)\cong\mathbb Z$.
The map~$\pi$ is~$R$-linear,
hence for~$p,q\in R$ with images~$\bar p,\bar q\in\mathbb Z$ in~$R/(f)$, we have~$\pi\bigl(p\,e'_1+q\,e'_2\bigr)=\bar p\,m+\bar q\,n.$
Since the images~$\bar p,\bar q$ range over all of~$\mathbb Z$,  the surjectivity of~$\pi$ says~$\mathbb Zm+\mathbb Zn=\mathbb Z$, i.e.~$\gcd(m,n)=1$.
Choose integers~$r,s$ with~$rm+sn=1$ and set
\[
 e_1:=r\,e'_1+s\,e'_2,
 \qquad
 e_2:=-n\,e'_1+m\,e'_2.
\]
The change of basis matrix
$\left(\begin{smallmatrix}r&-n\\s&m\end{smallmatrix}\right)$ lies in
$\SL_2(\mathbb Z)\subset\SL_2(R)$ because its determinant is~$rm+sn=1$.
It follows that~$e_1,e_2$ is again an~$R$-basis of~$E_1$, and
\[
 \pi(e_1)=rm+sn=1,
 \qquad
 \pi(e_2)=-nm+mn=0.
\]
We now conclude by calculating $\ker(\pi)=\im(\alpha)$.
An arbitrary element of~$E_1$ can be written as~$p(u)e_1+q(u)e_2$ with~$p(u),q(u)\in R$.
Its image under~$\pi$ is the residue
class of~$p(u)$ modulo~$f$, and this vanishes exactly when~$f$ divides
$p(u)$. 
It follows that~$\im(\alpha)=\ker(\pi)=fRe_1\oplus Re_2.$
This concludes the proof of the proposition.
\end{proof}

We can now prove the main theorem which states that there is a~$Q\in\SL_2(\La)$ with~$ \sigma(Q)=SQ$.

\begin{proof}[Proof of Theorem~\ref{thm:splitting}.]
Recall that constructing a matrix~$Q \in SL_2(\Lambda)$ with~$\sigma(Q)=SQ$ reduces to constructing vectors~$q_1 \in E_1$ and~$q_2 \in E_2$ with~$\det(q_1,q_2)=1$.
Recalling the $R$-basis~$e_1,e_2$ of~$E_1 \cong R^2$ from Proposition~\ref{prop:ImageAlpha}, we will see that setting~$q_1:=e_1$ and~$q_2:=\varepsilon \alpha^{-1}(e_2)$ will work for an appropriate choice of sign~$\varepsilon = \pm 1$.
Here,  we used Proposition~\ref{prop:ImageAlpha} to ensure that~$e_2$ lies in the image of~$\alpha$.
As we have already explained,  the fact that~$q_1 \in E_1$ and~$q_2 \in E_2$ ensures that the matrix~$Q=(q_1,q_2)$ satisfies~$\sigma(Q)=SQ$,  it only remains to prove that~$\det(q_1,q_2)=1$.
By definition of~$\alpha$,  we have~$\varepsilon e_2=\alpha(q_2)=h_-q_2$ so~$\varepsilon h_- \det(q_1,q_2)=\det(e_1,e_2)$ and we are reduced to showing that~$\det(e_1,e_2)=\varepsilon h_-$.
\begin{claim}
Taking the determinant induces an~$R$-linear isomorphism
\[
 \delta \colon \textstyle E_1 \wedge_R E_1 \longrightarrow h_-R,  \quad
 v \wedge w \mapsto \det(v,w).
\]
\end{claim}
\begin{proof}
We first argue that~$\delta$ does indeed take values in~$h_-R$.
Given~$v,w\in E_1$,  recall from~\eqref{eq:Econditions} that~$D\ol v=tSv$.
Form the matrix~$V:=(v,w)$, so that~$D\ol{V}=tSV$, set~$d:=\det(V)$ and take determinants to obtain~$\det(D)\ol d=\det(tS)\,d.$
This can be rewritten as~$t\ol d=t^2d$, i.e.~$d=t^{-1}\ol d$.
This implies that~$d \in (1+t^{-1})R=h_-R$, as required; see Lemma \ref{lem:eigenspaces}.

Next, we show that the map~$\delta$ is an isomorphism.
We will do this by showing that the induced map~$\delta_{\mathfrak{m}} \colon (E_1 \wedge_R E_1)_{\mathfrak m} \to  (h_-R)_{\mathfrak{m}}$ is an $R_{\mathfrak{m}}$-isomorphism for every maximal ideal $\mathfrak{m} \subset R$.
Exterior powers commute with localisation, so
$(E_1 \wedge_R E_1)_{\mathfrak m}=(E_1)_{\mathfrak m} \wedge_{R_{\mathfrak{m}}} (E_1)_{\mathfrak m}$.
Using the basis of~$q^{\mathfrak{m}}_1,h_-q^{\mathfrak{m}}_2 \in (E_1)_\mathfrak{m} \cong R_{\mathfrak{m}}^2$ from Proposition~\ref{prop:Free}, this has basis~$q^{\mathfrak{m}}_1\wedge(h_-q^{\mathfrak{m}}_2)$.
Recalling the matrix~$Q_{\mathfrak m}:=(q^{\mathfrak{m}}_1,q^{\mathfrak{m}}_2)$ from Theorem \ref{thm:local}, we then calculate
\[
\delta_{\mathfrak{m}} \bigl(q^{\mathfrak{m}}_1\wedge(h_-q^{\mathfrak{m}}_2)\bigr)
 =\det(q^{\mathfrak{m}}_1,h_-q^{\mathfrak{m}}_2)
 =h_-\det(Q_{\mathfrak m}).
\]
The unit~$\det(Q_{\mathfrak m}) \in \Lambda_{\mathfrak{m}}$  lies in~$R_{\mathfrak{m}}$ because $\det(Q_{\mathfrak m})
=\det(SQ_{\mathfrak m})
=\det(\sigma(Q_{\mathfrak m}))
=\overline{\det(Q_{\mathfrak m})}.$
Since~$\det(Q_{\mathfrak m})h_-$ generates~$(h_-R)_{\mathfrak m}=h_-R_{\mathfrak m}$,  this shows that~$\delta_{\mathfrak m}$ sends a basis to a basis and is thus an isomorphism, for every~$\mathfrak m$.  
The kernel and cokernel of~$\delta$ therefore localise to zero at every maximal ideal of $R$,  and therefore vanish.
Thus~$\delta$ is an $R$-isomorphism.
The claim follows.
\end{proof}

Since the map~$\delta$ is an isomorphism,~$\det(e_1,e_2)=\delta(e_1\wedge e_2)$ is a generator of~$h_-R$, i.e. for some unit~$\varepsilon\in R^\times$, we have~$ \det(e_1,e_2)=\varepsilon h_-.$
The only units of~$\mathbb Z[u]$ are~$\pm1$,
so~$\varepsilon\in\{\pm1\}$.
As explained above, setting~$q_1:=e_1 \in E_1$ and~$q_2:=\varepsilon \alpha^{-1}(e_2) \in E_2$ now yields~$\varepsilon h_-\det(q_1,q_2)= \det(e_1,e_2)= \varepsilon h_-$ and thus $\det(q_1,q_2)=1$, as required.
The matrix $Q:=(q_1,q_2) \in SL_2(\Lambda)$ now satisfies~$\sigma(Q)=SQ$ and this concludes the proof of Theorem \ref{thm:splitting}.
\end{proof}

\appendix

\section{Properties of the fixed ring~$\Z[u] \subset \Z[t^{\pm 1}]$}
\label{sec:fixed}

Write $\Lambda:=\Z[t^{\pm 1}]$ and recall from Section~\ref{sec:cocycle} that we set 
\[
 x=t-1,\qquad 
 u=x\overline{x}=2-t-t^{-1}.
\]
This appendix collects some facts about the subring $R=\Z[u] \subset \Lambda$.
For brevity, we write

\[
h_+:=1+t,\qquad
 h_-:=1+t^{-1}.
\]
The following lemma is elementary but is used throughout the article; its proof is left to the reader.

\begin{lemma}\label{lem:eigenspaces}
The ring~$\La$ is a free $R$-module with basis~$1,t$ and the following equalities hold:
\begin{align*}
 \{z\in\La:\ol z=z\}&=R,\\
 \{z\in\La:z=t\ol z\}&=h_+R,\\
 \{z\in\La:z=t^{-1}\ol z\}&=h_-R,\\
 \{z\in\La:z+t\ol z=0\}&=xR.
\end{align*}
The same assertions hold after localisation at an arbitrary multiplicative
subset of~$R$.
\end{lemma}

\bibliographystyle{alpha}
\bibliography{BiblioHomotopyBoundary}

\end{document}